\documentclass[11pt]{elsarticle} % Do USE A4  option; 10-point type preferred

\usepackage{fullpage}
\usepackage{amsfonts} 
\usepackage{graphics}

\usepackage{here} %pour forcer latex a placer image ou on veut avec option [H]

\usepackage{xr} % pouvoir referencer a autre document

\usepackage{refcheck} %pour infos sur references/citation utilisees ou pas
\norefnames	%pas de labels en marge
\nocitenames 	%pas de ref biblio en marge
\ignoreunlbld %pas de test label equations
\setoffmsgs %pas de msg du package en log

\usepackage{url}      % ecriture url dans refs

\usepackage{graphicx}
\usepackage{amsmath, amsthm, amssymb}
\usepackage{amssymb,amsmath,graphics,color}

\usepackage{epsfig}
\usepackage{epstopdf}

\usepackage{multirow}

\usepackage{relsize} %pour pouvoir agrandir en math mode avec \mathlarger and \mathsmaller

\usepackage{xcolor}
\def\red #1{\textcolor{red}{#1}}
\def\blue #1{\textcolor{blue}{#1}}

\def\purp #1{\textcolor{purple}{#1}}

\def\black #1{\textcolor{black}{#1}}

\definecolor{darkbrown}{rgb}{.3,.1,.1}

\def\coltab{\black}            % couleur cicrucit/cocircut dans tab fond
\def\ptt #1{{\small\tt #1}}    % taille non-cicrucit/cocircut dans tab fond

\allowdisplaybreaks[4] 

\usepackage{tikz}

\def\boxit [#1]#2{\vbox{\hrule\hbox{\vrule
     \vbox spread #1{\vss\hbox spread#1{\hss #2\hss}\vss}%
        \vrule}\hrule}}

\newbox\algo

\font\algofont=cmtt10 scaled 1100

\newenvironment {algorithme}{\smallskip \smallskip
\bgroup\algofont\parindent= 0mm}{\egroup
\smallskip\smallskip}

\newtheorem{thm}{Theorem}[section]
\newtheorem{cor}[thm]{Corollary} % share numbers with theorems
\newtheorem{prop}[thm]{Proposition} % share numbers with theorems
\newtheorem{lemma}[thm]{Lemma}
\newtheorem{remark}[thm]{Remark}

\newtheorem{property}[thm]{Property}
\newtheorem{observation}[thm]{Observation}
\newtheorem{definition}[thm]{Definition}
\newtheorem{example}[thm]{Example}
\def\newline{\hfill\break}

\def\newpage{\vfill\break}

\def\square{\hbox{\vrule\vbox{\hrule\phantom{o}\hrule}\vrule}}
\def\square{\hbox{\vrule\vbox{\hrule\phantom{O}\hrule}\vrule}}

\definecolor{shadecolor}{RGB}{0,0,150}

\def\Int{\mathrm{Int}}

\def\Ext{\mathrm{Ext}}

\def\min{\mathrm{min}}

\def\max{\mathrm{max}}

\def\ass{\mathrm{Part}}

\def\F{{\cal F}}
\def\X{\bullet}
\def\A{\underline{\mathrm{acl}}}
\def\AA{\mathrm{acl}}

\def\ep{\varepsilon}
\def\io{\iota}

\def\s{\setminus}
\def\bk{\backslash}

\def\ss{\smallskip}
\def\finsi{}

\def\plus{\uplus} % union disjointe

\def\M{M}

\def\bysame{\rule{1cm}{0.1mm}}

\def\eme#1{}
\def\emeref#1{}
\def\emenew#1{}

\def\emevder#1{}

\title{The Active Bijection 
\\ 
2.a -
Decomposition of activities for matroid bases, %\red{or bipartite graphs}, 
\\ and Tutte polynomial of a matroid in terms of beta invariants of minors}

\author[lirmm]{Emeric Gioan\corref{cor1}}
\ead{emeric.gioan@lirmm.fr}

\author[paris]{Michel Las Vergnas\corref{cor2}}

\cortext[cor1]{Email: emeric.gioan@lirmm.fr}
\cortext[cor2]{R.I.P.}
\address[lirmm]{CNRS, LIRMM, Universit\'e de Montpellier, France}
\address[paris]{CNRS, Paris, France}

\date{}

\begin{document}

\begin{abstract} % Short abstract
\noindent
We introduce and study filtrations of a matroid on a linearly ordered ground set, which are particular sequences of nested sets. A given basis can be decomposed into a uniquely defined sequence of bases of minors, such that these bases have an internal/external activity equal to $1/0$ or $0/1$ (in the sense of Tutte polynomial activities).
This decomposition, which we call the active filtration/partition of the basis, refines the known partition of the ground set into internal and external elements with respect to a given basis.
It can be built by a certain closure operator, which we call the active closure. It relies only on the fundamental bipartite graph of the basis and can be expressed also as a decomposition of general bipartite graphs on a linearly ordered set of vertices.
\emevder{laisser les "which we call" ? ou les enlever ?}

%From this, first, we get that the set of all bases can be partitioned into sets associated with filtrations, and these sets are decomposed into disjoint sets of such bases of minors induced by filtrations. %
From this, first, structurally, we obtain that the set of all bases can be canonically partitioned and decomposed 
in terms of 
%into disjoint sets of 
such bases of minors induced by filtrations. %\red{induced by the set of all filtrations.}
Second, enumeratively,\emevder{OU numerically, }
we derive an expression of the Tutte polynomial of a matroid in terms of beta invariants of minors. This expression refines at the same time the classical expressions in terms of basis activities and orientation activities (if the matroid is oriented), and the well-known convolution formula for the Tutte polynomial.
Third, in a companion paper of the same series (No. 2.b), we use this decomposition of matroid bases,  along with a similar decomposition of oriented matroids, and along with a bijection in the $1/0$ activity case from a previous paper (No. 1), to define the canonical active  bijection between orientations/signatures/reorientations and spanning trees/simplices/bases of a graph/real hyperplane arrangement/oriented matroid, as well as various related bijections.

\end{abstract}

\maketitle   % This line must be here, following the abstract
% Note that the Plain TeX commands \over and \atop do not work in
% AMS-LaTeX; fractions must be specified as \frac{1}{2} and 
% binomial coefficients as \binom{n}{k}

%\normalsize
%%%%%%%%%%%%%%%%%%%%%%%%%%%%%%%%%%%%%%%%%%%%%%%%%%%%%%%%%%%%%%%%%%

%\red{PENSER A VERIFIER LES FICHEIRS "A FAIRE" !!!}

%\red{penser a verifier les slash-eme}

%\red{AJOUTER FORMULE A 4 PARAMTERES ?}

\emevder{faire corrections d'anglais, comme dans chapter, voir fichier anglais.tex  passage "SUR CHAPTER"}%\\

\section{Introduction}
\label{sec:intro}

\eme{a mettre ? Erratum : phd, FPSAC 03 + mention : supersolv, LP}%

\eme{REMPLACER LE PLUS POSSIBLE on a ienarly ordered element set PAR ordered, MEME DANS INTRO}%

This paper studies some structural and enumerative properties of matroids on a linearly ordered ground set.
We introduce and study filtrations of a matroid on a linearly ordered ground set, which are simple particular sequences of nested subsets of the ground set (Definition~\ref{def:general-filtration}). They induce particular sequences of minors by the following manner: for each subset in the sequence, we consider the minor obtained by restriction to this subset and contraction of the subsets it contains.

A given basis can be decomposed into a uniquely defined sequence of bases of such minors (Theorem \ref{thm:unique-dec-seq-bas}), such that these bases have an internal/external activity equal to $1/0$ or $0/1$, in the sense of Tutte polynomial activities, as introduced by Tutte in \cite{Tu54}.
This decomposition can be seen as a partition that refines the known partition of the ground set into internal and external elements with respect to a given basis, as defined by Etienne and Las Vergnas in \cite{EtLV98}.
We call this unique special filtration/partition the active filtration/partition of the basis.

%From a constructive viewpoint,
%the construction 
From a constructive viewpoint,
%Regarding the construction, 
%it essentially  consists in 
%the active filtration 
it can be built by
applying a certain closure operator, which we call the active closure, to the internally/externally active elements of the basis, by several equivalent possible manners which are detailed in the paper (including notably a simple single pass over the ground set). This construction only  relies upon the fundamental bipartite graph of the basis and can be also expressed  as a decomposition of bipartite graphs on a linearly ordered set of vertices.
%From this, first, we get that the set of all bases can be partitioned into sets associated with filtrations, and these sets are decomposed into disjoint sets of such bases of minors induced by decomposing sequences. %

At a global level, we obtain that the set of all bases can be canonically partitioned and decomposed 
in terms of 
%into disjoint sets of 
such uniactive internal/external bases of minors induced by all filtrations, which is the main result of the paper (Theorem \ref{th:dec_base}).%\red{induced by the set of all filtrations.}
\bigskip

As the enumerative counterpart of the above structural decomposition theorem, we derive an expression of the Tutte polynomial of a matroid in terms of beta invariants of minors (Theorem \ref{th:tutte}):\break
%
%Let us give a glimpse of this formula here:
$$t(M;x,y)= \ \ \sum \ \ \Bigl(\prod_{1\leq k\leq \io}
\beta \bigl( M(F_k)/F_{k-1}\bigr)\Bigr)
 \ \Bigl(\prod_{1\leq k\leq \ep}\beta^* \bigl( M(F'_{k-1})/F'_{k}\bigr)\Bigr)\ {x^\io\  y^\ep}$$
\noindent 
where
%$\beta^*$ of a matroid equals $\beta$ of its dual
$\beta^*$ equals $\beta$ of the dual (that is $\beta$ except for an isthmus or a loop),
and where the sum is over all (connected) filtrations 
$\emptyset= F'_\ep\subset...\subset F'_0=F_c=F_0\subset...\subset F_\io= E$
%$(F'_\ep, \ldots, F'_0, F_c , F_0, \ldots, F_\io)$ 
of $M$. %(Definition \ref{def:general-filtration}).
\emevder{donner def de filtration dans intro ?}%
The beta invariant $\beta(M)$  of a matroid $M$ is equal of the coefficent of $x$ in the Tutte polynomial $t(M;x,y)$. It was specifically considered and so named by Crapo in \cite{Cr67}. In particular, it counts the number of bases having an internal/external activity equal to $1/0$ (or also $0/1$ as soon as the matroid has at least two elements) with respect to any linear ordering of the ground set.
%As a remark, 
%it is also known to count 
It also remarkably counts
the number of bounded regions of a real hyperplane arrangement 
(bipolar orientations in digraphs), 
as shown by Zaslavsky in \cite{Za75} and generalized to oriented matroids by Las Vergnas in \cite{LV77}
(see also \cite{GrZa83}, 
%see \cite{Za75,LV77, GrZa83} 
and see \cite{AB1, AB2-b} for the 
%bijective 
connection with bases, or \cite{GiLV05, ABG2} in graphs).
%a result from  \cite{Za75} that extends to oriented matroids \cite{LV77} and 
%has various interpretations and implications, e.g. see \cite{GrZa83, GiLV05, AB1, ABG2, AB2-b}.
\emevder{au debut j'avai splus detaille historique de beta orientations, mais enleve, ca donnait ca
***UTILE ?***
It also counts the number of bounded regions of a real hyperplane arrangement (regions that do not touch any given hyperplane, on one side of this hyperplane) or, in particular, the number of bipolar orientations of a graph
(acyclic orientations with adjacent unique source and unique sink, with fixed orientation for this source-sink edge), as shown by Greene and Zaslavsky in \cite{GrZa83}\red{REF???? plutot zaslavsky\cite{Za75}} \purp{refs utiles a orientations?} (a result extended to oriented matroids by Las Vergnas in \cite{LV77}, see also \cite{GiLV05} for more details in the context of graphs, and \cite{AB1, AB2-b} for more details in the context of hyperplane arrangements or oriented matroids).
}

%\red{let us anticipate the paper ??? and state this result, as it gives a glimpse of the stuctural involved cosntructions. ICI mettre thm}

%
%This expression refines at the same time the classical expressions of the Tutte polynomial in terms of basis activities (Tutte \cite{Tu54}) and orientation activities (if the matroid is oriented), and the so-called convolution formula for the Tutte polynomial.
The above expression of the Tutte polynomial in terms of beta invariants of minors %alluded to above 
thus refines at the same time the following known Tutte polynomial formulas:
\emevder{pourrait etre bien dans cette intro de redonner ces formules ? let us hughlight the enumerative coutnerpart in this introdiuction, though the paper is structural}
\begin{itemize}[-]
\partopsep=0mm \topsep=0mm \parsep=0mm \itemsep=1mm
\item The classical expression of the Tutte polynomial of a matroid in terms of basis activities, given by Tutte in \cite{Tu54} and extended to matroids by Crapo in \cite{Cr69} (recalled in Section \ref{sec:prelim} as the  \ref{eq:basis-activities} formula).
Indeed, by this classical expression, each coefficient of the Tutte polynomial counts the number of bases with given internal/external activity. 
By the above expression, 
%we can express further 
each coefficient of the Tutte polynomial is decomposed further
in terms of numbers of bases of minors with internal/external activity equal to $1/0$ or $0/1$ (see also Theorem \ref{th:dec_base} and the proof of Theorem \ref{th:tutte} at the very end of the paper).
%By the above expression, we can consistently express each coefficient of the Tutte polynomial in terms of numbers of bases of minors with internal/external activity equal to $1/0$ or $0/1$ (see also Theorem \ref{th:dec_base}).
\emevder{pas tres bien dit ?}

\item The expression of the Tutte polynomial of an oriented matroid in terms of orientation activities, given by Las Vergnas in  \cite{LV84a} (recalled in \cite[Section \ref{b-sec:prelim}]{AB2-b} as the  \ref{b-eq:reorientation-activities} formula).
Indeed, by this expression, each coefficient of the Tutte polynomial amounts to count the number of reorientations with given dual/primal orientation activity. 
By the above expression, 
%we can express further 
each coefficient of the Tutte polynomial 
is decomposed further 
in terms of numbers of reorientations of minors with dual/primal orientation activity equal to $1/0$ or $0/1$, that is 
in terms of numbers of bounded regions in minors of the primal and the dual with respect to a topological representation of the oriented matroid. See \cite{AB2-b} for details, notably \cite[Theorem \ref{b-th:dec-ori} and Remark \ref{b-rk:tutte}]{AB2-b}.
 
% recalled in \cite[Section \ref{b-sec:prelim}]{AB2-b}, as it gives an interpretation of the Tutte polynomial coefficients in terms of primal/dual bounded regions of minors, see \cite[Theorem \ref{b-th:dec-ori} and Remark \ref{b-rk:tutte}]{AB2-b}; %\red{update xxxxx}

%- and the next formula 
%in Corollary \ref{cor:convolution}, which was implicit in \cite{EtLV98} thanks to an explicit bijection (also refined in Section \ref{sec:dec-bases}), and called  ``convolution formula for the Tutte polynomial'' in \cite{KoReSt99}.

\item The convolution formula for the Tutte polynomial, recalled here as Corollary \ref{cor:convolution}, so named by Kook, Reiner and Stanton 
% by Kook, Reiner and Stanton 
 in \cite{KoReSt99}. 
This formula was implicit in \cite{EtLV98}, as it  is 
%can be seen also as 
a direct enumerative corollary of the structural decomposition of the set of bases into bases of minors with internal/external activity equal to zero, given by Etienne and Las Vergnas in \cite{EtLV98} (recalled here as Corollaries \ref{cor:dec-cyclic-flat} and \ref{th:EtLV98}).
One retrieves this formula from the above by considering only the subsets $F_c$ in the filtrations.
It expresses the Tutte polynomial in terms of Tutte polynomials of minors where either the variable $x$ or the variable $y$ is set to zero.
By the above expression, each Tutte polynomial of a minor involved in the convolution formula is further decomposed by means of a sequence of minors, thus using only the beta invariant of these minors (that is only the monomials $x$ or $y$ of the Tutte polynomial of these minors).
% 
% recalled in next Corollary \ref{cor:convolution}, and which was also implicit in \cite{EtLV98} thanks to an explicit bijection (a construction refined in Section \ref{sec:dec-bases}).\red{mal dit}
\end{itemize}

Let us  mention that an algebraic proof of the expression of the Tutte polynomial in terms of beta invariants of minors of Theorem \ref{th:tutte} 
could be obtained using the algebra of matroid set functions, a technique introduced %by Lass 
by Lass in \cite{La97}, 
according to its author \cite{Lass-perso}. 

%\red{debut section, intro...?eviter trop de repetitions ! donner noms propres quelque part...}
\bigskip

Finally, in the companion paper \cite{AB2-b},  No. 2.b of the same series, we use the above structural decomposition theorem of matroid bases (Theorem \ref{th:dec_base}),  along with a similar decomposition of oriented matroids (namely \cite[Theorem \ref{b-th:dec-ori}]{AB2-b}), and along with a bijection in the $1/0$ activity case from a previous paper, No. 1 \cite{AB1} (recalled in \cite[Section \ref{b-sec:bij-10}]{AB2-b}), to define the canonical active  bijection between orientations/signatures/reorientations and spanning trees/simplices/bases of a graph/real hyperplane arrangement/oriented matroid, as well as related bijections.

%The idea of decomposing activities of a basis had been introduced by an algorithm in \cite{LV84b}.
%
%Etienne and Las Vergnas in \cite{EtLV98}
%
%Kook, Reiner and Stanton in \cite{KoReSt99}
%
%polyn :
%Tutte in \cite{Tu54}
%Crapo in \cite{Cr69}
%
%beta : crapo
%
%Lass \red{dans intro, mieux en fait !}
%
%
%Those decomposition techniques use some particular sequences of nested subsets of the element set, or equivalently some particular partitions of the element set, called active partitions, associated with an oriented matroid on a linearly ordered element set (or with a base in a matroid) on a linearly ordered element set.
%%
%
%\bigskip
%
%This paper is an intermediate step in a series of papers on the active bijection,
%whose general setting  is to study oriented matroids on a linearly ordered set of elements (or directed graphs on a linearly ordered set of edges, or hyperplane arrangements on a linearly ordered set of hyperplanes), in terms of structural properties, constructions,  enumerative properties, invariants and bijections.
%
%This paper deserves a separate handling...
%We separate it from since it is interesting in terms of general matroids.

In brief, the active bijection for graphs, real hyperplane arrangements and oriented matroids  (in order of increasing generality)
is a framework introduced and studied in a series of papers by the present authors.
%In general, t
The canonical active bijection associates an oriented matroid on a linearly ordered ground set with one of its  bases. This defines an activity preserving correspondence between reorientations and bases of an oriented matroid, with numerous related bijections, constructions and characterizations. It yields notably a structural and bijective interpretation of the equality of the two expressions of the Tutte polynomial alluded to above: \ref{eq:basis-activities} by Tutte \cite{Tu54} and  \ref{b-eq:reorientation-activities} by Las Vergnas \cite{LV84a}.

%Let us mention that 
\emevder{OU AVANT j'avais mis : The idea of decomposing bases as in the present paper has been inspired by an algorithm from Las Vergnas' paper \cite{LV84b} ... OU AUSSI Let us mention that the decomposition of matroid bases developed in the present paper has been initiated by an algorithm}
The idea of decomposing matroid bases developed in the present paper 
has been initiated 
by an algorithm by Las Vergnas in \cite{LV84b} 
%(in graphs, given without proof, 
(given in graphs without proof, 
and allegedly yielding a correspondence between orientations and spanning trees, %in graphs 
different from the active bijection however, see \cite[footnote \ref{b-footnote:LV84}]{AB2-b}).
Most of the main results in this series (including the present paper) were given in the Ph.D. thesis \cite{Gi02} in a preliminary form. 
%for a correspondence between spanning trees and orientations of graphs was proposed, with no proof.
A short summary of the whole series (including the above Tutte polynomial formula) has been given in \cite{GiLV07}.
In the present paper, we will refer only to the journal papers \cite{GiLV05, AB1, AB2-b} of this series, the reader may see the companion paper \cite{AB2-b} for a complete overview and for further references from the authors and from the literature.
%
%Let us mention that the results of the present paper have been proved in the Ph.D. thesis \cite{Gi02} in a preliminary form, and that the decomposition of activites was also mentioned in \cite{GiLV05}. 

The reader primarily interested in graph theory may also read \cite{ABG2}, that gives a complete overview of the active bijection in the language of graphs (in contrast with other papers of the series), as well as a proof of the above Tutte polynomial expression in terms of beta invariants of minors by means of decomposing graph orientations (as done in \cite{AB2-b} for oriented matroids), instead of decomposing bases/spanning trees (as done in the present paper for matroids). This is possible in graphs since they are orientable, but this is not possible in non-orientable matroids.

\newpage
\section{Preliminaries}
\label{sec:prelim}

%\red{beta : \cite{Cr67}\cite{GrZa83}\cite{LV77} }
%\red{a reordonner ?}

%\red{dire pour graphes et arrangements, et dire bipolar}

%\subsection{Generalities}

%\red{PEUT TRE MIEUX DE DEFINIR REORIENTATIONS $-_AM$ EN DSITINGUANT DEUX OPPOSEES... plutot que de s'emboruiller avc des sous-ensembles....}

%\purple{section deja reecrites adaptees ent erms de m.o.}

\noindent{\it Generalities.}

In the paper,  $\subseteq$ denotes the inclusion, $\subset$ denotes the strict inclusion,  and $\uplus$ (or $+$) denotes the disjoint union.
%If $\mathcal F$ is a set of subsets of $E$, then $\cup \mathcal F$ denotes the subset of $E$ obtained by taking the union of all elements of $\mathcal F$.
%
%
%\red{ordered partout ? raccorucirait !}%
Usually, $M$ denotes a matroid on a finite set $E$. 
See \cite{Ox92} for a complete background on matroid theory, 
notably see \cite[Chapter 5]{Ox92} for the translation in terms of graphs, 
and \cite[Chapter 6]{Ox92} for the translation in terms of representable matroids, point configurations or real hyperplane arrangements.
%\red{See also \cite{ABG2} for preliminaries close to those below using graph terminology.}
%
A matroid  $M$ on $E$ can be called \emph{ordered} when the set $E$ is linearly ordered. Then, the dual $M^*$ of $M$ is ordered by the same ordering on $E$.
A minor $M/\{e\}$, resp. $M\bk\{e\}$, for $e\in E$, can be denoted for short $M/e$, resp. $M\bk e$.
A matroid can be called \emph{loop}, or \emph{isthmus}, if it has a unique element and this unique element is a loop ($M=U_{1,0}$), or an isthmus ($M=U_{1,1}$), respectively.
An isthmus is also called a coloop in the literature.
\emevder{changer partout isthmus en coloop ?}
%\subsection{Further usual definitions and technicalities}

\emenew{dessous inutile je crois}%
%%Technically, we will use some classical properties of circuits and cocircuits in minors, coming from classical (oriented) matroid theory.
%Let us recall some classical properties of minors in matroids.
%%in Section \ref{sec:tutte},
%%we will extensively use properties of circuits and cocircuits in minors. 
%%%Let us finally recall some combinatorial technique, coming from classical (oriented) matroid theory.
%%So, let us recall some combinatorial technique, coming from classical (oriented) matroid theory. 
%For $F\subseteq E$, it is known that:
%circuits of $M(F)$ are circuits of $M$ contained in $F$; cocircuits of $M(F)$ are non-empty inclusion-minimal intersections of $F$ and cocircuits of $M$; circuits of $M/F$ are non-empty inclusion-minimal intersections of $E\s F$ and circuits of $M$ (that is inclusion-minimal subsets obtained by removing $F$ from circuits of $M$); cocircuits of $M/F$ are cocircuits of $M$ contained in $E\s F$.
%%\red{ICI OU DANS SECTION TUTTE ? VOIR SI C'est utilise dans section FOB}

%\red{ptes des bases $M(F)$ et $M/F$.}

Let us first recall some usual matroid notions.
%Let us first recall some terminology inherited from classical matroid theory.
A \emph{flat} $F$ of $M$ is a subset of $E$ such that $E\setminus F$ is a union of cocircuits; equivalently: if $C\setminus \{e\}\subseteq F$ for some circuit $C$ and element $e$, then $e\in F$; and equivalently: $M/F$ has no loop.
A \emph{dual-flat} $F$ of $M$ is a subset of E which is a union of circuits; equivalently: its complement is a flat of the dual matroid $M^*$;  equivalently: if $D\setminus \{e\}\subseteq E\s F$ for some cocircuit $D$ and element $e$, then $e\in E\s F$; and equivalently: $M(F)$ has no isthmus.
A \emph{cyclic-flat} $F$ of $M$ is both a flat and a dual-flat of $M$;
 equivalently: $F$ is a flat and $M(F)$ has no isthmus; or equivalently: $M/F$ has no loop and $M(F)$ has no isthmus.
 
%

%If  $\M=(V,E)$ is an oriented matroid whose underlying unoriented matroid is $M$, we call $\M$ an \emph{reorientation} of $M$.
%An reorientation $\M$ of $M$ is called acyclic
\bigskip

%\subsection{Matroid basis activities and fundamental graph/tableau settings}
%\subsection{Matroid basis activities}

\noindent{\it Activities of matroid bases.}

Let $M$ be an ordered matroid on $E$, and let $B$ be  a basis of $M$.
For $b\in B$, the \emph{fundamental cocircuit} of $b$ with respect to $B$, denoted $C_M^*(B;b)$, or $C^*(B;b)$ for short,  is the  unique cocircuit contained in $(E\s B)\cup\{b\}$.
For $e\not\in B$, the \emph{fundamental circuit} of $e$ with respect to $B$, denoted $C_M(B;e)$, or $C(B;e)$ for short,
 is the unique circuit contained in $B\cup\{e\}$.
Let $$\Int(B)=\Bigl\{\ b\in B\ \mid\ b=\ \min \ \bigl(\ C^*(B;b)\ \bigr)\ \ \Bigr\},$$
$$\Ext(B)=\Bigl\{\ e\in E\setminus B\ \mid\ e=\ \min \ \bigl(\ C(B;e)\ \bigr)\ \ \Bigr\}.$$
We might add a subscript as $\Int_M(B)$ or $\Ext_M(B)$ when necessary.
The elements of $\Int(B)$, resp. $\Ext(B)$, are called \emph{internally active}, resp. \emph{externally active}, with respect to $B$. The cardinality of $\Int(B)$, resp. $\Ext(B)$ is called \emph{internal activity}, resp. \emph{external activity}, of $B$. 
We might write that a basis is \emph{$(i,j)$-active} when its internal and external activities equal $i$ and $j$, respectively.
Observe that $\Int(B)\cap \Ext(B)=\emptyset$ and that, for $p=\min (E)$,  we have $p\in \Int(B)\cup \Ext(B)$.

Moreover, let $B_{\min}$ be the smallest (lexicographic) base of $M$. 
 Then, as well-known and easy to prove, we have
  $\Int(B_{\min})=B_{\min}$, $\Ext(B_{\min})=\emptyset$, and $\Int(B)\subseteq B_{\min}$ for every base $B$.
Also, let $B_{\max}$ be the greatest (lexicographic) base of $M$. Then $\Int(B_{\max})=\emptyset$, $\Ext(B_{\max})=E\s B_{\max}$, and $\Ext(B)\subseteq E\s B_{\max}$ for every base $B$.
Thus, roughly, internal/external activities can be thought of as situating a basis with respect to $B_{\min}$ and $B_{\max}$.
Finally, we recall that internal and external activities are dual notions:
%$\Int_M(B)=\Ext_{M^*}(E\s B)$ and $\Ext_M(B)=\Int_{M^*}(E\s B)$.
$$\Int_M(B)=\Ext_{M^*}(E\s B) \ \ \text{ and }\ \  \Ext_M(B)=\Int_{M^*}(E\s B).$$
%If $b=\min\ C^*(B;b)$ then $b$ is \emph{internally active} in $B$.
%The set of internally active elements of $B$ is denoted $\Int(B)$,
%the cardinality of this set is the \emph{internal activity} of $B$.
%If $e=\min\ C^*(B;e)$ then $e$ is \emph{externally active} in $B$.
%The set of externally active elements of $B$ is denoted $\Ext(B)$,
%the cardinality of this set is the \emph{external activity} of $B$.

By \cite{Tu54, Cr69}, the Tutte polynomial of $M$ is
\begin{equation*}
\tag{``enumeration of basis activities''}
\label{eq:basis-activities}
%$$
t(M;x,y)=\sum_{\io,\ep}b_{\io,\ep}x^\io y^\ep
%$$ 
\end{equation*}
where $b_{\io,\ep}$ is the number of bases of $M$ 
with internal activity $\io$ and external activity $\ep$.
It does not depend on the linear ordering of $E$.

%See Figure \ref{fig:K4exbase256} for an example.
\bigskip

%\subsection{Comparing matroid basis activities and oriented matroid activities, beta invariant}
%\subsection{Beta invariant in matroids}

%\red{ajouter def uniactive internal bases!!!}

%\red{remplacer t par b dans enumeration bases}
Now, given a basis $B$ of $M$, if $\Int(B)=\emptyset$, resp.
$\Ext(B)=\emptyset$, then $B$ is called \emph{external}, resp. \emph {internal}.
If 
%$\mid \Int(B)\cup \Ext(B)\mid=1$ 
$\Int(B)\cup \Ext(B)=\{p\}$
then $B$ is called \emph{uniactive}.
Hence, a base with internal activity $1$ and external activity $0$ can be called uniactive internal, and a base with internal activity $0$ and external activity $1$ can be called uniactive external.
Let us mention that exchanging the two smallest elements of $E$ yields a canonical bijection between uniactive internal and uniactive external bases, see \cite[Proposition 5.1 up to a typing error%
\footnote{Let us correct here an unfortunate typing error in \cite[Proposition 5.1 and Theorem 5.3]{AB1}. The statement has been given under the wrong hypothesis 
$B_{\min}=\{p<p'<\dots\}$ 
%$B_{\min}=\{p<p'<\dots\}$ 
instead of the correct one $E=\{p<p'<\dots\}$. Proofs are unchanged
(independent typo: in line 10 of the proof of Proposition 5.1, instead of $B'-f$, read $(E\setminus  B')\setminus\{f\}$).
In \cite[Section 4]{GiLV05}, the statement of the same property in graphs is correct.
% An erratum has been published \com{***to be done***}.
}%
]{AB1}%
, see also \cite[Section 4]{GiLV05} in graphs.
\eme{remplacer partout ?}%
See the beginning of Section \ref{sec:dec-bases} for a reformulation of the characterization of uniactive internal/external bases (see also \cite[Proposition 2]{GiLV05} for another characterization, not used in the paper).\emevder{J'avais mis au debut de section 4 qu'il serait bien de remttre cette caraterrisation quelque part, mais maintenant j'an doute... a trancher !}

%\red{reordonner 1,0 tout apres ?}

%The \emph{beta invariant} of $M$ is defined by: $$\beta(M)=b_{1,0}.$$
In particular, by the above formula, we have that  $b_{1,0}$
%It 
counts the number of uniactive internal bases. 
This number  does not depend on the linear ordering of the element set $E$. 
This value $$\beta(M)=b_{1,0}$$ is known as the \emph{beta invariant} of $M$ \cite{Cr67}. Assuming $\mid E\mid >1$, it is known that $\beta(M)\not=0$ if and only if $M$ is connected. Let us recall that, %in terms of a graph $G$, 
for a loopless graph $G$ with at least three vertices, the associated matroid $M(G)$ is connected if and only if $G$ is 2-connected. 
Also, we have $\beta(M)=b_{1,0}=b_{0,1}=\beta(M^*)$ as soon as   $\mid E\mid >1$.
Note that, assuming $\mid E\mid=1$, we have $\beta(M)=1$ if the single element is an isthmus of $M$,
and $\beta(M)=0$ if the single element is a loop of $M$.

%\red{$\beta(M)=b_{1,0}=b_{0,1}=\beta(M^*)$ as soon as   $\mid E\mid >1$.}
%
%Finally, we need to introduce a slight variant of the parameter $\beta$, namely $\beta^*$, defined, for every matroid $M$, by $$\beta^*(M)=\beta(M^*).$$
%In fact, we have $\beta^*(M)=\beta(M)$ for a matroid $M$ with at least two elements, and we have $\beta^*(M)=1-\beta(M)$ for a matroid with one element: $\beta^*$ of a matroid with one element equals $1$ if this element is a loop and equals $0$ is this element is an isthmus, and this is the only change compared with $\beta$.
%
Finally, for our constructions, we need to introduce the following dual slight variation $\beta^*$ of $\beta$:
% $\beta^*$ that satisfies:
$$
 \beta^*(M)=\beta(M^*)=b_{0,1}=\ %={o_{0,1}\over 2}= \
\Biggr\{
\begin{array}{ll}
       \beta(M) &\text{ if }|E|>1 \\
       0 &\text{ if $M$ is an isthmus} \\
       1 &\text{ if $M$ is a loop.} 
\end{array}
$$

%\red{charact de bases 1,0 aussi ? vu qu'on charact bounded... ou bien referer a Property dans section decomp bases}
\bigskip

%\bigskip
%\subsection{Fundamental bipartite graph/tableau settings}
%\label{subsec:prelim-fund}
\noindent{\it Fundamental bipartite graph/tableau settings.}

\emenew{variatne dessous}%
%In the paper, the constructions rely in fact on bipartite graphs (fundamental graphs of bases) and can be applied in this context, regardless of the matroid setting. Let us detail this, so that we can state further results in this setting, as it can be combinatorially interesting on its own.

%\paragraph{Fundamental bipartite graph/tableau settings}
Observe that the above definitions for a basis $B$ of an ordered matroid $M$ only rely upon the fundamental circuits/cocircuits of the basis, not on the whole structure $M$. 
In the paper, we develop a combinatorial construction that also only depends on this local data, and thus can be naturally expressed in terms of general bipartite graphs on a linearly ordered set of vertices. So let us introduce the following definitions and representations.
This is rather formal but necessary.

We call \emph{(fundamental) bipartite graph} $\F$ on $(B,E\s B)$ a bipartite graph on a set of vertices $E$, which is bipartite w.r.t. a couple of subsets of $E$ forming a bipartition $E=B\uplus E\s B$. 
%It is \emph{ordered} when $E$ is linearly ordered. 
We call \emph{(fundamental) tableau} $\F$ on $(B,E\s B)$ a matrix whose rows and columns are indexed by $E$, with entries in $\{\X,0\}$, and such that each diagonal element indexed by $(e,e)$, $e\in E$, is non-zero and, moreover, is the only non-zero entry of its row (when $e\in B$), or the only non-zero entry of its column (when $e\in E\s B$).
We use the same notation $\F$ for a bipartite graph or a tableau since,
obviously, bipartite graphs and tableaux are equivalent structures: each non-diagonal entry of the tableau represents an edge of the corresponding bipartite graph.
We choose to define both because graphs are the underlying compact combinatorial structure, whereas tableaux are better for visualization, notably for signs of the fundamental circuits/cocircuits in the oriented matroid case developed in the companion paper \cite{AB2-b}, and they are consistent with the matrix representation used in the linear programming setting of the active bijection developed 
in \cite{AB1}. 
%in \cite{AB1,AB3}. 
In what follows (and in \cite{AB2-b} too), examples will be illustrated  on both representations. 

 \begin{figure} %[htbh]
\centering
\includegraphics[scale=1.2]{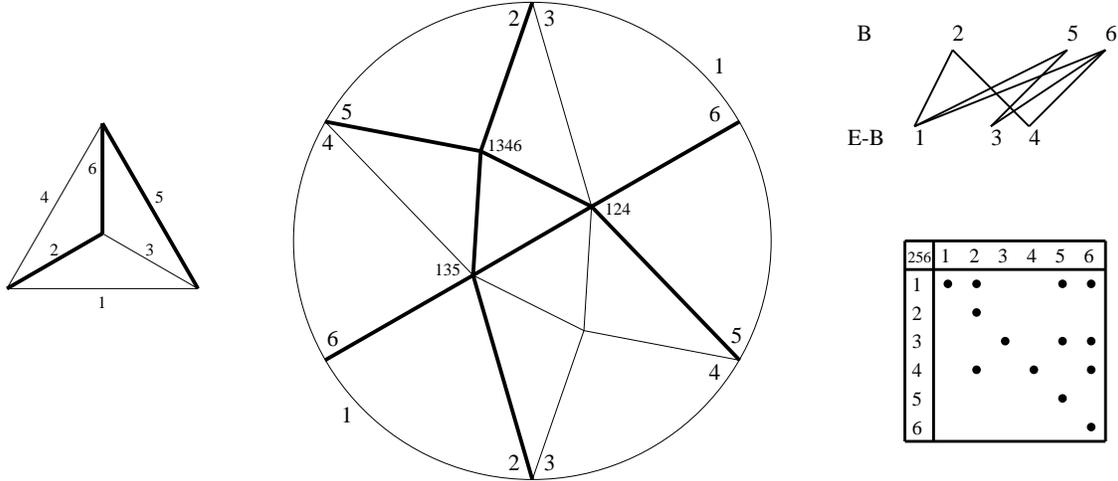}
\caption{For the base $256$ of the depicted matroid of %the represented graph
 $K_4$ with ground set $1<\dots<6$, we have $\Int(256)=\emptyset$ and $\Ext(256)=\{1,3\}$. On the left: a graph representation. In the middle: a hyperplane arrangement representation (we represent $\min(E)$ as a hyperplane at infinity, and we only represent one half of the arrangement, on a given side of $\min(E)$, see \cite[Section 2]{AB1} for more details on such representations); and the vertices associated with fundamental cocircuits of the basis. On the upper right and the bottom right, respectively: the fundamental bipartite graph and the fundamental tableau of the basis (see last part of Section \ref{sec:prelim}).
%\red{a detailler, expliquer representztion hyperplans, ref a AB1, et dire que fund graph/tableau on the right see subsection 2.3}
}
\label{fig:K4exbase256}
\end{figure}

Given a basis $B$ of a matroid $M$ on $E$, the \emph{fundamental graph/tableau} of $B$ in $M$, denoted $\F_M(B)$ is the usual (fundamental) bipartite graph/tableau on $(B,E\s B)$ such that, for every $b\in B$, $b$ is adjacent to elements of $C^*(B;b)\s \{b\}$, 
and  for every $e\in E\s B$, $e$ is adjacent to elements of $C(B;e)\s \{e\}$.
Recall that $$e\in C^*(B;b)\text{ if and only if }b\in C(B;e).$$
%
%The reader may think of a (fundamental) bipartite graph as the classical fundamental graph of a basis $B$ in a matroid $M$ on $E$: for every $b\in B$, $b$ is adjacent to elements of $C^*(B;b)\s \{b\}$.
%
In fact, every bipartite graph on $(B,E\s B)$ is the fundamental graph of some basis $B$ in some matroid $M$ on $E$: one just has to choose $B$ as a vector space basis, and settle elements $e$ of $E\s B$ in general position in the subspaces spanned by the elements of $B$ adjacent to $e$ (isthmuses correspond to isolated vertices in $B$ and loops correspond to isolated vertices in $E\s B$).

Observe that matroid duality comes down to exchange the roles of $B$ and $E\s B$, that is to exchange the two parts of the bipartition of $E$  (in the bipartite graph setting), or to transpose the matrix (in the tableau setting).
Precisely, for a bipartite graph/tableau $\F$ on $(B,E\s B)$, we define the dual $\F^*$ of  $\F$ as the bipartite graph/tableau on $(E\s B,B)$ with same edges/transposed values w.r.t. $\F$. Obviously, for a basis of a matroid $M$,  we have $\F_{M^*}(E\s B)=(\F_M(B))^*$.

Thus, the usual definitions and notations given above can be directly extended to bipartite graphs: for $b\in B$, $C^*(B;b)$ is the set of elements adjacent to $b$, plus $b$; for $e\in E\s B$, $C(B;e)$ is the set of elements adjacent to $e$, plus $e$; and, assuming $E$ is linearly ordered, an element is internally, resp. externally, active if it is in $B$ and it is the  smallest element of $C^*(B;b)$, resp. the smallest element of $C(B;e)$.
Similarly, those definitions translate in the tableau setting:
for $b\in B$, $C^*(B;b)$ is defined by the non-zero entries of the column indexed by $b$, or column $b$ for short; for $e\in E\s B$, $C(B;e)$ is defined by the non-zero entries of the row indexed by $e$, or row $e$ for short; and, assuming $E$ is linearly ordered, an element is internally, resp. externally, active if its corresponding diagonal element is the smallest non-zero entry of its column, resp. its row.
Then we can directly extend the notations  $\Int$ and $\Ext$, and the relative definitions, to those settings.

Finally, for $A\subseteq E$, we define $\F-A$ as the bipartite graph/tableau obtained by removing all vertices (and their incident edges)/lines in $A$ from $\F$.
For an element $e\in E$, we can denote $\F-e$ instead of $\F-\{e\}$. 
%\red{ici ou ailleurs?}

%\bigskip

\begin{example}
\label{ex1}
{\rm
An example of a matroid basis, its internal/external activities, its (fundamental) bipartite graph and its (fundamental) tableau is given in Figure \ref{fig:K4exbase256}. Internal/external activities for all bases of this example are listed in Figure \ref{fig:tabK4} at the end of the paper.
}
\end{example}

%\red{*** AJOUTER FIGURE}

%\red{def fund graph/tableau and defs acts}

%\subsection{Oriented matroid activities, comparison with basis activities, and beta invariant}

%
%At last, we recall that, if $D'$ is a cocircuit of the minor $\M/B\setminus A$ of $\M$, then there exists a unique %\red{***pourquoi unique ?!***}
%cocircuit $D$ of $\M$ such that $D\cap B=\emptyset$ and $D\setminus A=D'$, and then $D$ has same signs as $D'$ on the elements of $D'$. We say that $D'$ is \emph{induced} by $D$.
%\red{OU DANS PREUVE SECTION 4}

%\red{dire grpahe vocabulaire (strongly connected, bounded, ptet circuits... et hyperplane arrangement, pour jsutifier titre !!!!}

%\red{****???????  prendre cyclic-bounded ou dual-bounded ????****}

%%%%%%%%%%%%%%%%%%%%%%%%%%%%%%%%%%%%%%%%%%%%%%%%%%%%%%%%%%%%%%%%%%%%%%%%%%%%%%%%%%%%%%%

\section{Filtrations of an ordered matroid, and Tutte polynomial in terms of beta invariants of minors induced by filtrations}
\label{sec:dec-seq}

%\new{attention (aussi dans autres artiocles) isthmus et loop sont connected... faut distinguer cardinal, ou bien dire non-isthmus non-loop --- cf def donnee dans chapter TB}

%\new{attention ? cette condition  << $\beta(M)\not= 0$ if and only if $M$ is not a loop and is connected (i.e. loopless $2$-connected for a  graph) >> dans autres articles a cause de loop}

%\new{filtration / connectde flitration}

%\new{verifier avec AB2-a}

First, we introduce filtrations of a matroid on a linearly ordered ground set, which are particular increasing sequences of subsets of the ground set and which will be continuously used throughout the paper.
Then, we introduce a formula for the Tutte polynomial of a matroid in terms of beta invariants of minors induced by filtrations. Its proof will be given at the very end of the paper, as a consequence of the structural decomposition of matroid bases with respect to basis activities, developed in the next section. Let us mention that, in the particular case of oriented matroids (or real hyperplane arrangements, or graphs, whose associated matroids are all orientable), this formula can be equally proved using a decomposition of oriented matroids with respect to orientation activities, using the same filtrations,
see 
%as developed in 
\cite{AB2-b} (or \cite{ABG2} in graphs).
%We choose to present this formula and its context at first, despite its abstract formalism, since the same sequences of subsets will be used concretely in the next two sections, allowing, because of their similarity, to extend the active bijection in the ultimate section.
%\red{mettre numeros des sections}
%\red{+ mentonner preuve dans autres papiers ? et preuve ent erms de matroides/oriented/graphs}
%\bigskip

%We recall that $\subset$ denotes a strict inclusion.
%

%\bigskip
%
%\blue{***deplace vers intro***
%We will also extensively use properties of circuits and cocircuits in minors. So let us recall that, for $F\subseteq E$:
%circuits of $M(F)$ are circuits of $M$ contained in $F$; cocircuits of $M(F)$ are intersections of $F$ and cocircuits of $M$; circuits of $M/F$ are interesctions of $E\s F$ and circuits of $M$ (or circuits of $M$ minus $F$); cocircuits of $M/F$ are cocircuits of $M$ contained in $E\s F$.
%}
%
%\red{***enlever les ''***}

\begin{definition}
\label{def:general-filtration}
Let $E$ be a linearly ordered finite set.
Let $M$ be a matroid on $E$.
%Let $M$ be a matroid on a linearly ordered set $E$.
  We call \emph{filtration of $M$} (or $E$) %(or of $M$) 
  a 
sequence $(F'_\ep, \ldots, F'_0, F_c , F_0, \ldots, F_\io)$ of subsets of $E$
 such~that:%
\begin{itemize}
\itemsep=0mm
\partopsep=0mm 
\topsep=0mm 
\parsep=0mm
\item $\emptyset= F'_\ep\subset...\subset F'_0=F_c=F_0\subset...\subset F_\io= E$;
\item the sequence $\min(F_k\setminus F_{k-1})$, $1\leq k\leq\io$  is increasing with $k$;
\item the sequence  $\min(F'_{k-1}\setminus F'_k)$, $1\leq k\leq\ep$, is increasing with $k$.
\end{itemize}
%We call \emph{filtration of $M$} %\red{***ou valid***} 
%an abstract filtration $(F'_\ep, \ldots, F'_0, F_c , F_0, \ldots, F_\io)$ of $M$ such that:
The sequence is a \emph{connected filtration of $M$} if, in addition:
\emevder{OU : A filtration of $M$ is called \emph{connected (w.r.t. $M$)} if, in addition:}%
\begin{itemize}
\vspace{-1mm}
\itemsep=0mm
\partopsep=0mm 
\topsep=0mm 
\parsep=0mm
%\item for every $0\leq k\leq\io$, the subset $F_k$ is a flat of $M$;
%\item for every $1\leq k\leq\io$, the minor $M(F_{k})/F_{k-1}$ is either a single isthmus, or is connected; %\red{****************?}
%\item for $1\leq k\leq\io$, the minor $M(F_{k})/F_{k-1}$ is connected and is not a loop (it can be a single isthmus);
\item for $1\leq k\leq\io$, the minor $M(F_{k})/F_{k-1}$ is connected and is not a loop;
%\item for every $0\leq k\leq\ep$, the subset $F'_k$ is a dual-flat of $M$;
%\item for $1\leq k\leq\ep$, the minor $M(F'_{k-1})/F'_k$ is connected and is not an isthmus (it can be a single loop).
\item for $1\leq k\leq\ep$, the minor $M(F'_{k-1})/F'_k$ is connected and is not an isthmus.
%\red{**********************?}
%\item for every $1\leq k\leq\ep$, the minor $M(F'_{k-1})/F'_k$ is either a single loop, or is connected. %\red{**********************?}
%\item the subset $F_c$ is a cyclic-flat of $M$ (this condition is implied by the previous ones).
\end{itemize}
\end{definition}

%

%\red{OBSERVATION SUR DUALITE ET SEQUENCES INDUITES}

In what follows, we can equally use the notations $(F'_\ep, \ldots, F'_0, F_c , F_0, \ldots, F_\io)$
or $\emptyset= F'_\ep\subset...\subset F'_0=F_c=F_0\subset...\subset F_\io= E$ to denote a filtration of $M$.
The $\io+\ep$ minors involved in Definition \ref{def:general-filtration} are said to be \emph{associated with} or \emph{induced by} the filtration.
The subset $F_c$ will be called \emph{the cyclic-flat} of the  
%connected 
filtration when it is connected 
(a term justified by Lemma \ref{lem:des-seq-flats} below).
Observe that filtrations of $M$ are equivalent to pairs of partitions of $M$ formed by a bipartition obtained from the subset $F_c$ (with possibly one empty part, which is a slight language abuse) and a refinement of this bipartition:
$$E=F_c\uplus E\s F_c,$$
$$E= (F'_{\ep-1}\s F'_\ep)\ \uplus\ \dots \ \uplus\ (F'_{0}\s F'_1)\ \uplus\ (F_1\s F_0)\ \uplus\ \dots \ \uplus\ (F_{\io}\s F_{\io-1}).$$
Indeed, one can retrieve the sequence of nested subsets from the pair of partitions since  the subsets in the sequence are unions of parts given by the ordering of the smallest elements of the parts.

%
%
%To motivate this definition, 
%Recall that 
%
%Let us recall that, for a matroid $M$ with at least two elements, $M$ is connected if and only if $\beta(M)\not= 0$.
%
%Let us also recall that a matroid with only one element $p$ is bounded w.r.t. $p$ if $p$ is an isthmus, and cyclic-bounded w.r.t. $p$ if $p$ is a loop.

%\red{**** ATTENTION REMPLACEMENT FFAITS depuis ABG2 graph en matroid --- VERIFIER DEFS AVEC CONNECTED ETC !!!!}

%\medskip

The next Lemma \ref{lem:dec-seq-beta} is used in the Tutte polynomial formula below.

\begin{lemma}
\label{lem:dec-seq-beta}
Let $M$ be an ordered matroid on $E$.
A filtration $\emptyset= F'_\ep\subset...\subset F'_0=F_c=F_0\subset...\subset F_\io= E$ of $M$ is a connected filtration of $M$ if and only if
$$\Bigl(\prod_{1\leq k\leq \io}
\beta \bigl( M(F_k)/F_{k-1}\bigr)\Bigr)
 \ \Bigl(\prod_{1\leq k\leq \ep}\beta^* \bigl( M(F'_{k-1})/F'_{k}\bigr)\Bigr)\ \not=\ 0.$$%
\end{lemma}
%\end{lemma}

\begin{proof}
The result is direct.
For a matroid $M$ with at least two elements, we have $\beta(M)\not=0$ if and only if $M$ is connected, and, according to Section \ref{sec:prelim}, we have $\beta(M)=\beta^*(M)$. Moreover, we have $\beta(M)=1$ and $\beta^*(M)=0$ if $M$ is an isthmus, and $\beta(M)=0$ and $\beta^*(M)=1$ if $M$ is a loop.
%follows directly from the above Definitions \ref{def:abst-dec-seq} and \ref{def:matroid-dec-seq}.
\end{proof}

We give the next Lemma \ref{lem:des-seq-flats} for the intuition and information, but it is  not practically used thereafter. % in what follows.

\begin{lemma}
\label{lem:des-seq-flats}
A connected filtration $(F'_\ep, \ldots, F'_0, F_c , F_0, \ldots, F_\io)$ of an ordered matroid $M$ satisfies:%
\vspace{-1mm}
\begin{itemize}
\itemsep=0mm
\item for every $0\leq k\leq\io$, the subset $F_k$ is a flat of $M$;
%\item for every $1\leq k\leq\io$, the minor $M(F_{k})/F_{k-1}$ is either a single isthmus, %\red{necessaire?}, 
%or is loopless and 2-connected; %\red{****************?}
\item for every $0\leq k\leq\ep$, the subset $F'_k$ is a dual-flat of $M$;
%\item for every $1\leq k\leq\ep$, the minor $M(F'_{k-1})/F'_k$ is either a single loop, or is loopless and 2-connected; %\red{**********************?}
\item the subset $F_c$ is a cyclic-flat of $M$. 
%(this condition is implied by the previous ones).
\end{itemize}
\end{lemma}

\begin{proof}
%Assume $F_k$ is not a flat, for some $0\leq k\leq \io$.
Assume there exists $k$, $1\leq k\leq \io$, such that $F_{k-1}$ is not a flat.
By definition, there exists an element $e$ and a circuit  $C$ of $M$ such that $e\not\in F_{k-1}$ and $C\s \{e\}\subseteq F_{k-1}$.
Let $j$ be the largest integer such that $e\not\in F_{j-1}$. We have $j\geq k$, $C\s \{e\}\subseteq F_{j-1}$ since $j\geq k$, 
$e\not\in F_{j-1}$, and $e\in F_{j}$ by maximality of $j$.
So, $C\s F_{j-1}=\{e\}$ is a circuit of $M(F_j)/F_{j-1}$, 
that is $e$ is a loop of $M(F_j)/F_{j-1}$, contradiction. 
%We get a contradiction with Definition \ref{def:matroid-dec-seq}.
%So, $\beta (\ M(F_j)/F_{j-1}\ )=0$.

Dually, assume there exists $k$, $1\leq k\leq \ep$, such that $F'_{k-1}$ is not a dual-flat.
%\red{***pty de dual-flat----OK***}
%if $D\setminus \{e\}\subseteq E\s F$ for some cocircuit $D$ and element $e$, then $e\in E\s F$.
By definition, there exists an element $e$ and a cocircuit  $D$ of $M$ such that $e\in F'_{k-1}$ and $D\s \{e\}\subseteq E\s F'_{k-1}$.
Let $j$ be the largest integer such that $e\in F'_{j-1}$. We have $j\geq k$, $D\s \{e\}\subseteq E\s F'_{j-1}$ since $j\geq k$, 
$e\in F'_{j-1}$, and $e\not\in F'_{j}$ by maximality of $j$.
So, $D\cap F'_{j-1}=\{e\}$ is a cocircuit of $M'(F_{j-1})/F'_{j}$, that is $e$ is an isthmus of $M'(F_{j-1})/F'_{j}$, contradiction. 
%So, $\beta^* (\ M'(F_{j-1})/F'_{j}\ ) =0$.
%We get a contradiction with Definition \ref{def:matroid-dec-seq}.

Finally $F_c=F=0=F'_0$ is a cyclic flat as it is both a flat and a dual-flat.
%
%
%\eme{dual part  a VERIFIER----- OK a l'air OK}
%
%\eme{attenion et si une partie duale est une boucle ???}
\end{proof}

%\red{ou dessous lemma, ou dessus observation?}

%\begin{lemma}
\begin{observation}
\label{lem:dec-seq-observation}
Let $\emptyset= F'_\ep\subset...\subset F'_0=F_c=F_0\subset...\subset F_\io= E$ be a connected filtration of an ordered matroid $M$. We have the following properties.
\begin{itemize}
\itemsep=0mm
\partopsep=0mm 
\topsep=0mm 
\parsep=0mm
\item $\emptyset= E\s F_\io\subset...\subset E\s F_0=E\s F_c=E\s F'_0\subset...\subset E\s F'_\ep= E$ is a connected filtration of~$M^*$, for the cyclic-flat $E\s F_c$ of $M^*$.
\item The minors associated with the above  filtration of $M^*$ are the duals of the minors associated with the above filtration of $M$. That is, precisely:
for every $1\leq k\leq\io$, 

%\begin{displaymath}
%\[
%\begin{equation}
%\begin{align}
\centerline{$\bigl( M(F_{k})/F_{k-1}\bigr)^*=M^*(E\s F_{k-1})/(E\s F_k),$}
%\vspace{-2mm}
%$$\bigl( M(F_{k})/F_{k-1}\bigr)^*=M^*(E\s F_{k-1})/(E\s F_k),$$
%\end{align}
%\end{equation}
%\]
%\end{displaymath}
and for every $1\leq k\leq\ep$, 

%\vspace{-2mm}
\centerline{$\bigl( M(F'_{k-1})/F'_k\bigr)^*=M^*(E\s F'_{k})/(E\s F'_{k-1}).$}
\item $\emptyset= F'_\ep\subset...\subset F'_0=F_c=F_c$ is a connected filtration of $M(F_c)$, for the cyclic-flat $F_c$ of $M(F_c)$.
\item $\emptyset=\emptyset= F_0\s F_c\subset...\subset F_\io\s F_c=E\s F_c$ is a connected filtration of $M/F_c$, for the cyclic-flat $\emptyset$ of $M/F_c$.
\end{itemize}
\end{observation}

\begin{thm}
\label{th:tutte}
Let $M$ be a matroid on a linearly ordered set $E$. We have
$$t(M;x,y)= \ \ \sum \ \ \Bigl(\prod_{1\leq k\leq \io}
\beta \bigl( M(F_k)/F_{k-1}\bigr)\Bigr)
 \ \Bigl(\prod_{1\leq k\leq \ep}\beta^* \bigl( M(F'_{k-1})/F'_{k}\bigr)\Bigr)\ {x^\io\  y^\ep}$$
% $$t(M;x,y)= \ \ \sum \ \ \Bigl(\prod_{1\leq k\leq \io}
%\bar\beta \bigl( M(F_k)/F_{k-1}\bigr)\Bigr)
% \ \Bigl(\prod_{1\leq k\leq \ep}\bar\beta \bigl( M(F'_{k-1})/F'_{k}\bigr)\Bigr)\ {x^\io\  y^\ep}$$
%}
%%%%%%
%DESSOUS version v3
%\noindent where the sum is over all possible active filtrations in $M$, and
%where, by convention, $\beta(M)=1$ if~\hbox{$\mid E\mid=1$.}
\noindent 
%%where, by convention, 
%%\red{ref a lemme pour $\beta^*$}
%where
%%$\bar\beta=1$ for a matroid with one element, and $\bar\beta=\beta$ otherwise,
%\red{xxx}
%$\beta^*=\beta$ for a matroid with at least two elements, and  $\beta^*=1-\beta$ for a matroid with one element ($\beta^*$ of a matroid with one element equals $1$ if this element is a loop and equals $0$ is this element is an isthmus),
%%single loop element equals $1$, $\beta^*=1$ for a matroid with a single isthmus element (that is $\beta^*=1-beta$ for a matroid with one element), 
%%and $\beta^*=\beta$ otherwise,
%%where
%%$\bar\beta(H)=\beta(H)$ if $\mid E(H)\mid>1$, and $\bar\beta(H)=1$ if~\hbox{$\mid E(H)\mid=1$,}
%and 
where the sum can be equally:
\begin{itemize}
\item over all connected filtrations 
%$(F'_\ep, \ldots, F'_0, F_c , F_0, \ldots, F_\io)$ 
$\emptyset= F'_\ep\subset...\subset F'_0=F_c=F_0\subset...\subset F_\io= E$
of $M$;
\item or over all filtrations 
$\emptyset= F'_\ep\subset...\subset F'_0=F_c=F_0\subset...\subset F_\io= E$
%$(F'_\ep, \ldots, F'_0, F_c , F_0, \ldots, F_\io)$ 
of $E$.
\end{itemize}
%\begin{itemize}
%\item either over all sequences of sets type
%$\emptyset= F'_\ep\subset...\subset F'_0=F_c=F_0\subset...\subset F_\io= E$
%where $F_c$ is a union of circuits \com{*?*}, the sequence $\min(F'_k)$, $1\leq k\leq\ep$  is increasing, and the sequence  $\min(F_k)$, $1\leq k\leq\io$, is decreasing,
%\item or over all possible active filtrations in $M$, that is sequences of the above type where each minor in the formula is loopless and $2-connected$.
%\end{itemize}
\end{thm}
%
%\com{*enoncer-comme-theoreme?*}

%
%The equality of the two possible sums is easy since  $\beta(M)\not= 0$ if and only if $M$ is not a loop and is connected (i.e. loopless $2$-connected for a  graph).
%Hence, non-zero terms in the second sum correspond to connected filtrations.%

The fact that the sum in Theorem \ref{th:tutte} can be equally made over the two types of sequences directly comes from Lemma \ref{lem:dec-seq-beta}: non-zero terms in the second sum correspond to connected filtrations. The proof that the sum yields the Tutte polynomial is postponed at the very end of
Section \ref{sec:dec-bases}, since it is derived from the main result of the paper, namely Theorem \ref{th:dec_base}.
%  which gives a constructive proof of the result by a decomposition of matroid bases, refining the identity \ref{eq:basis-activities}.
%An alternative similar proof in the particular case of an orientable matroid could be given using Theorem \ref{th:dec-ori} in Section \ref{sec:dec-mo} instead, as it provides a similar decomposition for oriented matroid reorientations, refining the identity \ref{eq:reorientation-activities}. This alternative proof is given in \cite{ABG2} in the case of graphs (which are all orientable).
%
%Those two decompositions and their similarity allow to extend the active bijection from bounded to general oriented matroids in Section \ref{sec:can-act-bij}.
See the introduction of the paper for comments on how the Tutte polynomial formula given in Theorem  \ref{th:tutte} refines other known formulas.
Let us detail in the corollary below how Theorem  \ref{th:tutte} refines the convolution formula for the Tutte polynomial.

\begin{cor}[\cite{EtLV98, KoReSt99}]
\label{cor:convolution}
Let $M$ be a matroid. We have
$$t(M;x,y)=\sum t(M/F_c;x,0)\ t(M(F_c);0,y)$$ where the sum can be either over all subsets $F_c$ of $E$, or over all cyclic flats  $F_c$ of $M$.
\end{cor}

\begin{proof}
By fixing $y=0$ in Theorem \ref{th:tutte}, we get

\centerline{$\displaystyle t(M;x,0)= \ \ \sum \ \ \Bigl(\prod_{1\leq k\leq \io}
\beta \bigl( M(F_k)/F_{k-1}\bigr)\Bigr)
 \  {x^\io}$}
\noindent where the sum is over all (connected) filtrations where the susbet $F_c$ satisfies $F_c=\emptyset$, that is of the type $\emptyset=F'_0= F_c = F_0\subset \ldots\subset F_\io=E$ of $M$.
By fixing $x=0$, we get

\centerline{$\displaystyle t(M;0,y)= \ \ \sum \ \ 
 \ \Bigl(\prod_{1\leq k\leq \ep}\beta^* \bigl( M(F'_{k-1})/F'_{k}\bigr)\Bigr)\ {y^\ep}$}
\noindent where the sum is over all (connected) filtrations where the susbet $F_c$ satisfies $F_c=E$, that is of the type $\emptyset=F'_\ep\subset \ldots\subset F'_0= F_c = F_0=E$ of $M$.
Then, by decomposing the sum in Theorem \ref{th:tutte} as $\sum_{F_c}\sum_{i,j}  \Pi_{1\leq k\leq \io}\dots\Pi_{1\leq k\leq \ep}\dots$, and by the fact that connected filtrations of $M/F_c$ and $M(F_c)$ are directly induced by that of $M$, as shown in Observation \ref{lem:dec-seq-observation}, 
we get the formula
$t(M;x,y)=\sum t(M/F_c;x,0)\ t(M(F_c);0,y)$ where the sum is over all cyclic flats  $F_c$ of $M$.
If $F_c$ is not a cyclic flat, then either $M/F_c$ has a loop or $M(F_c)$ has an isthmus, implying that the corresponding term in the sum equals zero.
\end{proof}

%The formula in Corollary \ref{cor:convolution} was implicit in \cite{EtLV98} thanks to an explicit bijection, and called  {\it convolution formula for the Tutte polynomial} in \cite{KoReSt99}. 
%As explained in the introduction of this section, the more general formula given by Theorem \ref{th:tutte} generalizes further to oriented matroids using a   generalization of the decomposition of reorientations, and more further to matroids using a decomposition of bases (bases in the matroid case) in terms of filtrations, that refines the decomposition of bases into internal/external bases \cite{EtLV98},
%see below. \red{***}
%%see \cite{AB2} for all the details (see also \cite{Gi02,GiLV05,GiLV07}). 
%Those proofs are constructive.
%%
%%Let us mention that an algebraic proof of Theorem \ref{th:tutte} formula could be obtained using matroid set functions \cite{Lass-perso}.
%Let us mention that an algebraic proof of the formula in Theorem \ref{th:tutte} can be obtained using matroid set functions, a technique introduced in \cite{La97}, 
%according to its author \cite{Lass-perso}.
%according to the author of \cite{La97}.
%%%%%%%%%%%%%%%%%%%%%%%%%%%%%%%%%%%%%%%%%%%%%%%%%%%%

%\section{Decomposition of activities for matroid bases (or for bipartite graphs)}
%\section{Decomposition of a matroid basis  into uniactive internal/external bases of minors (\red{active closure,} and decomposition of a general bipartite graph)}
\section{Decomposition of matroid bases  into uniactive internal/external bases of minors (and underlying decomposition of a general bipartite graph)}
\label{sec:dec-bases}

%\red{Recall $\F$... +few constructive lemmas, givzen for \Int but available for \Ext as well}

\emevder{OU (and related decomposition of a general bipartite graph)}%

We begin with giving some properties of the fundamental graph $\F_M(B)$ of a basis $B$ in a matroid~$M$.
Next, we define an \emph{active closure} operation that can be applied on such a fundametnal graph, and in fact on any (fundamental) bipartite graph/tableau (see end of Section \ref{sec:prelim}), as it depends only on this local graph, not on the whole matroid structure.
Next, we give a few useful combinatorial lemmas to characterize or to build this operation, they also only rely on the bipartite graph structure.
Then, we essentially apply this operation in a matroid setting to build decompositions of a matroid basis.
First, we recall and develop a decomposition into two so-called \emph{internal} and \emph{external} bases of minors, a construction introduced in \cite{EtLV98}. 
Finally, we build a decomposition, that refines the above one, into a sequence of uniactive internal/external bases of minors, in terms of connected filtrations introduced in Section \ref{sec:dec-seq}, yielding 
%a decomposition of a matroid basis  into uniactive internal/external bases of minors and,
 by the way a proof of Theorem \ref{th:tutte}.
\bigskip

%\new{autre caracterisation de internal uniactive: GiVL05 Prop 2, bien a remettre quelque aprt}

%\red{
%Let $E$ be a linearly ordered set.
%Let $\F$ be a bipartite graph/tableau on $(B,E\s B)$ (or equivalently let $B$ be a basis of a matroid $M$ on $E$ with fundamental graph $\F$).
Let us first recall that, from Section \ref{sec:prelim},
given a linearly ordered set $E$,
 a bipartite graph/tableau $\F$
on $(B,E\s B)$, or equivalently a basis $B$ of a matroid $M$ on $E$ with fundamental graph $\F=\F_M(B)$, is uniactive when the following property holds: for all $b\in B\s \min(E)$, we have $b\not=\min(C^*_M(B;b))$, that is $\min(C^*_M(B;b))\in E\s B$, and, moreover, for all $e\in (E\s B)\s \min(E)$, we have $e\not=\min(C_M(B;e))$, that is $\min(C_M(B;e))\in B$. Then, under these conditions, it is internal, resp. external, if $\min(E)$ is internally active, that is $\min(E)\in B$, resp. if $\min(E)$ is  externally active, that is $\min(E)\in E\s B$.
%}
%Recall also that $\subset$ denotes the strict inclusion.
%\red{def de uniactive et internal/external dans prelim}

%\begin{color}{gray}
%*** deja dans prelim ***
%
%Let $E$ be a finite totally ordered set, $B\subseteq E$, and $(G,B)$, abbreviated in $G$ (the definitions relative to $G$ depend on $B$), a bipartite graph between  $B$ and $E\s B$.
%
%By analogy, for $e\in B$ (resp. $e\not\in B$),
%we denote $C^*(B;e)$ (resp. $C(B;e)$) the set formed by  $e$ and the vertices adjacent to $e$ in $G$.
%Hence $e\in C^*(B;b)$ if and only if $b\in C(B;e)$, 
%and we say that $e$ is internally active (resp. externally active) if it is the minimal element of $C^*(B;e)$ 
%(resp. $C(B;e)$). We denote $\Int(G)$ (resp. $\Ext(G)$) the set of internally active 
%(resp. externally active) elements of $G$.
%Obviously by definition for a base  $B$ of a matroid $M$ we have
%$\Int_M(B)=\Int(\F_M(B))$ and $\Ext_M(B)=\Ext(\F_M(B))$.
%
%For $A\subseteq E$, we denote $(G-A,B\s A)$, abbreviated in $G-A$, the graph obtained by deleting from
%$G$ the vertices of $A$ and their incidental edges.\par
%\end{color}

%\red{The fundamental graph $\F$ of a basis $B$ of a matroid $M$ on $E$ is the bipartite graph $\F$ on $(B,E\s B)$. A DEFINIR TEL QUE DANS PRELIM}
%\red{attention notation $(B,E\s B)$ pas bonne, il manque aretes....!!!}

\begin{property} %[reformulation of Property \ref{pty:fund_circ_deletion}] %{55}
\label{pty:fund_graph}
Let $B$ be a basis of a matroid $M$ on $E$.
% and let $\F$ be the fundamental graph of $B$.\par
For $b\in B$, we have $$\F_M(B)-b=\F_{M/b}({B-b}).$$
For $e\in E\s B$, we have $$\F_M(B)-e=\F_{M\s e}({B}).$$
%For $b\in B$, $\F-\{b\}$ is the fundamental graph of the basis $B\s\{b\}$ of $M/b$.
%For $e\in E\s B$, $\F-\{e\}$ is the fundamental graph of the basis $B$ of $M\s e$.
%If $b\in B$ then $$\F_M(B)-b=\F_{M/b}({B-b})$$\par
%If $e\not\in B$ then $$\F_M(B)-e=\F_{M\s e}({B})$$\par
\endproof
\vspace{-7mm}
\hfill\square
\end{property}

%\begin{proof}
%
%\end{proof}

%
%\begin{color}{gray}
%\begin{property} %{60}
%\label{pty:fund_graph_flat}
%\red{a ton besoin de flat? on voudrait pouvoir deduire tableaux des bases externes aussi... A VOIR !!!}
%Let  $B$ be a basis of a matroid $M$ on $E$, and let $F$ be a flat of $M$.
%%\par
%If $B\cap F$ is a basis of $M(F)$ then 
%$B\s F$ is a basis of $M/F$, and we have
%$$\F_M(B)-F=\F_{M/F}({B\setminus F}),$$
%$$\F_M(B)-(E\s F)=\F_{M(F)}({B\cap F}).$$
%\end{property}
%\end{color}

%
%\begin{property} %{60}
%\label{pty:fund_graph_flat}
%\red{prendre celle-ci !}
%\red{ A T'ON BESOIND E FLAT ??? pas utilsie dans preuve !!!}
%Let  $B$ be a basis of a matroid $M$ on $E$.
%If $F$ is a flat of $M$ and $B\cap F$ is a basis of $M(F)$ then 
%$B\s F$ is a basis of $M/F$.
%If $F$ is a dual-flat of $M$ and $B\s F$ is a basis of $M/F$ then 
%$B\cap F$ is a basis of $M(F)$.
%In both cases, we have:
%%\par
%$$\F_M(B)-F=\F_{M/F}({B\setminus F});$$
%$$\F_M(B)-(E\s F)=\F_{M(F)}({B\cap F}).$$
%\end{property}
%

\begin{property} %{60}
\label{pty:fund_graph_flat}
Let  $B$ be a basis of a matroid $M$ on $E$. Let $F\subseteq E$.
The following properties are equivalent:
\begin{enumerate}[(i)]
\item  \label{pty-it1} $B\cap F$ is a basis of $M(F)$;
\item \label{pty-it2} $B\s F$ is a basis of $M/F$;
\item \label{pty-it3} for all $b\in B\s F$, we have $C^*_M(B;b)\cap F=\emptyset$;
\item \label{pty-it4} for all $e\in F\s B$, we have $C_M(B;e)\subseteq F$.
\end{enumerate}
If the above properties are satisfied, we have:
%\par
$$\F_M(B)-F=\F_{M/F}({B\setminus F});$$
$$\F_M(B)-(E\s F)=\F_{M(F)}({B\cap F}).$$
Moreover, if both $F\subseteq E$ and $G\subseteq E$ satisfy the above properties, and $F\subseteq G$, then
$B\cap (G\s F)$ is a basis of $M(G)/F$, and
$$\F_M(B)-\bigl((E\s F)\cup G\bigr)=\F_{M(G)/F}\bigl(B\cap (G\s F)\bigr).$$
%\red{ecrire equivalence avec ptes utilisees dans lemme decomposition}
\end{property}

%\begin{property} %{60}
%\label{pty:fund_graph_flat}
%Let  $B$ be a basis of a matroid $M$ on $E$. Let $F\subseteq E$.
%Assume that $B\cap F$ is a basis of $M(F)$, or, equivalently, that $B\s F$ is a basis of $M/F$.
%We have:
%%\par
%$$\F_M(B)-F=\F_{M/F}({B\setminus F});$$
%$$\F_M(B)-(E\s F)=\F_{M(F)}({B\cap F}).$$
%\red{ecrire equivalence avec ptes utilisees dans lemme decomposition}
%\end{property}

\begin{proof}
%\red{preuve a completer}
%Let us first check the equivalence given in the statement.
%Assume  $B\cap F$ is a basis of $M(F)$.
%It is well known that, for every basis $B'$ of $M(F)$, that $B\s F$ is a basis of $M/F$ if and only if its union with $B'$ is a basis of $M$.
%With $B'=B\s F$, since $(B\s F)\cup (B\cap F)=B$ is a basis of $M$, we have that $B\s F$ is a basis of $M/F$. The inverse implication is obtained by dualiy.
The fact that (\ref{pty-it1}) implies (\ref{pty-it2})  comes directly  from the following usual property:  for every basis $B'$ of $M(F)$,  $B\s F$ is a basis of $M/F$ if and only if $B'\uplus (B\s F)$ is a basis of $M$.
Then, the inverse implication comes directly from duality.
%The fact that (\ref{pty-it1}) implies (\ref{pty-it2})  comes directly  %from  equivalence between (\ref{pty-it1}) and (\ref{pty-it2})  comes directly from duality.
The equivalence between (\ref{pty-it3}) and (\ref{pty-it4}) comes directly from:
$b\in C_M(B;e)$ if and only if $e\in C^*_M(B;b)$. %\red{***preciser?}
The equivalence between (\ref{pty-it1}) and (\ref{pty-it4})  comes directly from the fact that the two properties are equivalent to: $B\cap F$ is a spanning set in $M(F)$ (since $B\cap F$ is independant in $M(F)$).
%With $B'=B\s F$, since $(B\s F)\cup (B\cap F)=B$ is a basis of $M$, we have that $B\s F$ is a basis of $M/F$. The inverse implication is obtained by dualiy.
%Then, the elements of $F\s B$ are loops in $M(F)/(B\cap F)$, and hence loops in $M/(B\cap F)$.
%Since $B$ is a basis of $M$,
%Then $B\s F$ is a maximal independent set of $M/F=M/(B\cap F)/(F\s B)$, that is a basis of $M/F$.

%Observe that property (\ref{pty-it1}) is stated as:
%for all $e\in E\s ((E\s F) \cup B)$, we have $C(B;e) \cap (E\s F)=\emptyset$.
%That is, equivalently:
%for all $e\in F\s B$, we have $C(B;e) \subseteq F$.
%That is, equivalently: $B\cap F$ is a spanning set in $M(F)$.
%That is, equivalently: $B\cap F$ is a basis of $M(F)$ (since $B\cap F$ is independant in $M(F)$).
%

Now let us assume that those properties are satisfied.
Since  $B\cap F$ is a basis of $M(F)$, the elements of $F\s B$ are loops in $M(F)/(B\cap F)$, and hence loops in $M/(B\cap F)$.
Contracting or deleting loops (or isthmuses) in a matroid yields the same result. Hence, $M/F=M/(B\cap F)\s (F\s B)$.
Hence,
with Property \ref{pty:fund_graph}, 
%(and the usual property that $M/A\s A'=M\s A'/A$ for all $A,A'\subseteq E$),
we get $\F_M(B)-F=\F_{M/F}({B\setminus F})$.
Now, if we delete from $M$ the elements of $(E\s F)\s B$, then the elements of
$(E\s F)\cap B$ become isthmuses and we conclude the same way to get $\F_M(B)-(E\s F)=\F_{M(F)}({B\cap F})$.
%The second case, where $F$ is a dual-flat, is dual to the first case, and yields the same result.

Finally, let us assume that the above properties are satisfied for $F$ and $G$ with $F\subseteq G\subseteq E$.
As seen above, we have that $B\cap G$ is a basis of $M(G)$.
We also have that for all $b\in B\s F$, we have $C^*_M(B;b)\cap F=\emptyset$.
This implies in particular that  for all $b\in B\cap G\s F$, we have $C^*_M(B;b)\cap G\cap F=\emptyset$. Since $B\cap G$ is a basis of $M(G)$, we obtain $C^*_{M(G)}(B\cap G;b)\cap G\cap F=\emptyset$. This implies, by the above equivalence applied to $B\cap F$ in $M(F)$, that $B\cap G\s F$ is a basis of $M(G)/F$.
\end{proof}

The above Property \ref{pty:fund_graph_flat} will often be used in what follows, possibly without reference, to translate properties from bipartite graphs to matroid bases and conversely, and to relate the fundamental circuits and
cocircuits of a basis in $M$ with those in some minors of type
 $M(F)$ or $M/F$.

\begin{definition}
\label{def:dec_operation}

Let $E$ be a linearly ordered set.
Let $\F$ be a bipartite graph/tableau on $(B,E\s B)$ (or equivalently let $B$ be a basis of a matroid $M$ on $E$ with fundamental graph $\F$).
For $b\in B$, and dually for $e\not\in B$, we denote $$C^*(B;b)^<=\{e\in C^*(B;b)\mid e<b\},$$
$$C(B;e)^<=\{b\in C(B;e)\mid b<e\}.$$
%
%\red{idem pour $C$}
Then, for $X\subseteq \Int(\F)$, we define the \emph{active closure} $\AA_\F(X)$ of $X$, or $\AA(X)$ for short, as the smallest subset of $E$ for inclusion such that:
\begin{itemize}
\topsep=0mm
\partopsep=0mm
\itemsep=0mm
\item $X\subseteq \AA(X)$;
\item if $b\in B$ and $b\in \AA(X)$ then $C^*(B;b)\subseteq \AA(X)$;
\item if $b\in B$ and $\emptyset\subset C^*(B;b)^<\subseteq \AA(X)$ then $b\in \AA(X)$ (and hence $C^*(B;b)\subseteq \AA(X)$).
\end{itemize}
And dually, for $X\subseteq \Ext(\F)$, we define the \emph{active closure} $\AA(X)$ of $X$ as the smallest subset of $E$ for inclusion such that:
%
%\red{ce dually a verifier}
\begin{itemize}
\topsep=0mm
\partopsep=0mm
\itemsep=0mm
\item $X\subseteq \AA(X)$;
\item if $e\in E\s B$ and $e\in \AA(X)$ then $C(B;e)\subseteq \AA(X)$;
\item if $e\in E\s B$ and $\emptyset\subset C(B;e)^<\subseteq \AA(X)$ then $e\in \AA(X)$ (and hence $C(B;e)\subseteq \AA(X)$).
\end{itemize}
\end{definition}

%\red{observe that $\AA(X)\cap \Int(\F)=X$.}

%\purple{AJOUETR def pour $Y\uplus Z$ ??? definir $\AA(X)=\AA(Y)\cup \AA(Z)$, et remarquer que $\AA(X)=\AA(Y)\uplus \AA(Z)$}

\begin{observation}
\label{obs:dual-closure}
As noted previously, the parts $B$ and $E\s B$ of $\F$ play dual parts,
as well as internally and externally active elements.
The definition of the active closue is consistent with this duality as we directly have that:
if $X\subseteq \Int(\F)$ then $X\subseteq \Ext(\F^*)$
%\red{definir ($\F^*$) dans prelim}
and 
$$\AA_\F(X)=\AA_{\F^*}(X).$$
Note that the %constructive combinatorial 
lemmas that follow are given in terms of internally active elements, but they can be stated dually as well, for externally active elements. We will focus on internally active elements and simply use duality to extend results.
\end{observation}

% \purple{and even for combinations of the two, see Lemma, but we do not develop this here ((in this paper))}.
%\red{dessous formualtion par puissance equivalente mais en trop ?}

We give Lemma \ref{lem:act-closure-iteration} below for consistency with the definition given in \cite[Section 5]{GiLV05}.

\begin{lemma}
\label{lem:act-closure-iteration}
Let $E$ be a linearly ordered set.
Let $\F$ be a bipartite graph/tableau on $(B,E\s B)$  (or equivalently let $B$ be a basis of a matroid $M$ on $E$ with fundamental graph $\F$).
%Let $\F=(B,E\s B)$ be a bipartite graph/tableau on a linearly ordered set $E$.
%
For $X\subseteq E$,
let $$\A(X)=X\ \cup\  \Biggl(\bigcup_{b\in X\cap B}C^*(B;b)\Biggr)\ \cup\ \Bigl\{b\in B\s X\mid \emptyset\subset
C^*(B;b)^<\subseteq X\Bigr\}.$$
Then, for $X\subseteq \Int(\F)$,
we have $$\AA(X)=\ \bigcup_{i\geq 1}\ \A^i(X).$$
\end{lemma}

\begin{proof}
It is a direct reformulation of Definition \ref{def:dec_operation}. %\red{******}
\end{proof}

%\red{\emph{active closure}}
%
%%\end{definition}
%
%\red{donner def duale de $\AA$ aussi !}

%\red{lemmas provide alternative definitions, they are easy reformulations}
The two next lemmas could be used as alternative definitions of the active closure. They are easy reformulations, and useful from a constructive viewpoint.

\begin{lemma}
\label{lem:act-closure-algo}
Let $E=e_1<...<e_n$ be a linearly ordered set.
Let $\F$ be a bipartite graph/tableau on $(B,E\s B)$ (or equivalently let $B$ be a basis of a matroid $M$ on $E$ with fundamental graph $\F$).
%
%Let $\F=(B,E\s B)$ be a bipartite graph/tableau on a linearly ordered set $E=e_1<...<e_n$.
Let $X\subseteq \Int(\F)$. Then $\AA(X)$ is given by the following % inductive
 definition (yielding a linear algorithm).

\begin{algorithme}
%First, set $\AA(X)=\emptyset$. \emph{(Note: in what follows, $e_i\in \AA(X)$ means that $e_i$ is added to $\AA(X)$)}\par
For all $1\leq i\leq n$:\par
%For $i$ from $1$ to $n$ do:\par
\hskip 1cm if $e_i\in X$ then $e_i\in \AA(X)$;\par
\hskip 1cm if $e_i\in B$ is not internally active \par
\hskip 15mm and if all $c\in C^*(B;e_i)$ with $c<e_i$
satisfies $c\in\AA(X)$, then $e_i\in\AA(X)$;\par
\hskip 1cm if $e_i\not\in B$ and there exists $c\in C(B;e_i)$ with $c<e_i$ and $c\in\AA(X)$ then $e_i\in\AA(X)$;\par
\hskip 1cm in every other case, $e_i\not\in \AA(X)$.
\end{algorithme}
\end{lemma}

\begin{proof}
Let $1\leq i\leq n$.
We analyze under which condition the element $e_i$ belongs to $\AA(X)$.
If $e_i\in X$ then $e_i\in \AA(X)$ directly by definition.
Let $e_i\in B\s X$. 
If $e_i$ is internally active, then $C^*(B;e_i)=\emptyset$, hence $e_i\not\in \AA(X)$, by definition.
Assume $e_i$ is not internally active.
We have $e_i\in\AA(X)$ if and only if $C^*(B;e_i)^<\subseteq \AA(X)$, that is 
if and only if, for all $c\in C^*(B;e_i)^<$, we have $c\in \AA(X)$, which is the condition given in the algorithm. 
Now let $e_i\not\in B$.
%
%Assume $e_i\in\AA(X)$.
Using the definition given in Lemma \ref{lem:act-closure-iteration}, 
we have $e_i\in\AA(X)$ if and only if
$e_i$ is added to $\AA(X)$ by $\A^j(X)$ for some (minimal) $j$, $e_i$ being an element of  $C^*(B;c)$ for some $c\in \A^{j-1}(X)\cap B$.
Such a $c$ satisfies $c<e$, since $C^*(B;c)^<\subseteq \A^{j-1}(X)$ and $e_i\not\in \A^{j-1}(X)$. And it satisfies $c\in C(B;e_i)$, as this property is equivalent to $e_i\in C^*(B;c)$. So we have that  $e_i\in\AA(X)$ if and only if there exists $c\in C(B;e_i)$, with $c<e_i$ and $c\in\AA(X)$.
%
%
%we have $e_i\in\AA(X)$ if and only if there exists $b\in \AA(X)\cap B$ such that $e_i\in C^*(B;b)$.
%On the other hand, 
%***********************************
%
%\red{a reprednre OLD}
%For (ii), first,
%$b\in B$ is in $A^i(X)$ by definition if an only if  $b\in X$ or $C^*(B;b)^<\subseteq A^i(X)$,
%so $b\in B$ is in $\AA(X)$ by definition if and only if $b\in X$ or $C^*(B;b)^<\subseteq \AA(X)$.
%On the other hand let $e\not\in B$. If $e\in A(X)$ then $e\in C^*(B;b)$ for $b\in X$
%which is internally active, so $b<e$. 
%If $e\in \AA(X)\s A(X)$ then let $i$ be the least element such that $e\in A^i(X)\s A^{i-1}(X)$.
%Then we have $e\in C^*(B;b)$ for $b\in A^{i-1}(X)$,
%and $b<e$, otherwise $e\in A^{i-1}(X)$.
%So $e\not\in B$ is in $\AA(X)$ if and only if it is in
%$C^*(B;b)$ for $b\in \AA(X)$ with $b<e$, that is if and only if there exists $b\in C(B;e)\cap \AA(X)$, $b<e$.
%According to the previous conditions, the fact that 
%an element belongs to  $\AA(X)$ or not depends only on this property for smaller elements,
%that's why we get this linear algorithm.
\end{proof}

\begin{lemma}
\label{lem:act-closure-algo-bas}
Let $E$ be a linearly ordered set.
Let $\F$ be a bipartite graph/tableau on $(B,E\s B)$ (or equivalently let $B$ be a basis of a matroid $M$ on $E$ with fundamental graph $\F$).
Assume $E=e_1<\dots<e_n$.
%
%Let $\F=(B,E\s B)$ be a bipartite graph/tableau on a linearly ordered set $E=e_1<...<e_n$.
Let $X\subseteq \Int(\F)$. Then $\AA(X)$ is given by the following % inductive
 algorithmic definition.

\begin{algorithme}
%First, set $\AA(X)=\emptyset$. \emph{(Note: in what follows, $e_i\in \AA(X)$ means that $e_i$ is added to $\AA(X)$)}\par
Initialize $\AA(X):=\emptyset$.\par
For $i$ from $1$ to $r$ do:\par
%For $i$ from $1$ to $n$ do:\par
%\hskip 1cm if $b_i\in X$ then $\AA(X):=\AA(X)\cup C^*(B;b_i)$;\par
%\hskip 1cm if $b_i$ satisfies $\emptyset\subset C^*(B;b_i)^<\subseteq \AA(X)$ then $\AA(X):=\AA(X)\cup C^*(B;b_i)$.
 \hskip 1cm if $b_i\in X$ or $b_i$ satisfies $\emptyset\subset C^*(B;b_i)^<\subseteq \AA(X)$ then $\AA(X):=\AA(X)\cup C^*(B;b_i)$.\par
%\hskip 1cm if $b_i$ satisfies $\emptyset\subset C^*(B;b_i)^<\subseteq \AA(X)$  then $\AA(X):=\AA(X)\cup C^*(B;b_i)$.
\end{algorithme}
\end{lemma}

\begin{proof}
This alternative formulation for a definition of $\AA$ is intermediate between the ones given in Definition \ref{def:dec_operation} and Lemma \ref{lem:act-closure-algo}.
The proof is straightforward.
\end{proof}

%\red{lemme dessous pourrait etre observation}
%\begin{lemma}
\begin{lemma}
\label{lem:closure-easy}
%Using the above notations,
Let $E$ be a linearly ordered set.
Let $\F$ be a bipartite graph/tableau on $(B,E\s B)$ (or equivalently let $B$ be a basis of a matroid $M$ on $E$ with fundamental graph $\F$).
Let $X\subseteq \Int(\F)$. % (or  $X\subseteq \Ext(\F)$). 
We have $$\AA(X)\cap \bigl(\Int(\F)\cup \Ext(\F)\bigr)=X.$$

In particular, if $\AA(X)=E$ then $X=\Int(\F)$, % (or $X=\Ext(\F)$),
and if $\AA(\{x\})=E$ for $x\in \Int(\F)$ %(or $x\in X=\Ext(\F)$) 
then $\F$ is uniactive internal. % (or external).
%We have:
%\red{*****en bleu: a mettre ?**** a mettre des def de $\AA$?*** lemme vraimetn utile ??? ***en fait corollaire de partition itnerne/externe... utilise uniquement dans preuve de prop principale, mais pas indispensable A VOIR AVEC CETTE RPEVUE}
%\begin{itemize}
%\begin{color}{blue}
%\item for $X\subseteq \Int(\F)$, we have $\AA(X)=E$ if and only if $X=\Int(\F)$ and $\F$ is internal (i.e. $\Ext(\F)=\emptyset$);
%\item for $X\subseteq \Ext(\F)$, we have $\AA(X)=E$ if and only if $X=\Ext(\F)$ and $\F$ is internal (i.e. $\Int(\F)=\emptyset$);
%\end{color}{blue}
%\item for $x\in \Int(\F)$, resp. $x\in \Ext(\F)$, we have $\AA(\{x\})=E$ if and only if $x=\min(E)$ and $\F$ is uniactive internal, resp. external.
%\end{itemize}
\end{lemma}

\begin{proof}
Direct by Lemma \ref{lem:act-closure-algo}: if $e_i$ is internally active and $e_i\not\in X$ then $e_i\not\in \AA(X)$; and if $e_i$ is externally active then there exist no $c<e_i$ with $c\in C(B;e_i)$, and then $e_i\not\in \AA(X)$.
\end{proof}

We give Lemma \ref{lem:act-closure-part} below for practical purpose.
It notably shows that the active closure of $X\subseteq \Int(\F)$ can be computed using  active closures of its elements, successively in any order, while deleting successively the results from $\F$.
%\red{ In what follows, it will be practically used in this restricted form, $Y$ being successively the greatest internally active element of $\F$.}
%\red{en fait non, pas utilise, sauf pour construction directe de dec seq (observation) !}

\begin{lemma}
\label{lem:act-closure-part}
Let $E$ be a linearly ordered set.
Let $\F$ be a bipartite graph/tableau on $(B,E\s B)$ (or equivalently let $B$ be a basis of a matroid $M$ on $E$ with fundamental graph $\F$).
%Let $\F=(B,E\s B)$ be a bipartite graph/tableau on a linearly ordered set $E$.
Let $X\subseteq \Int(\F)$ and let $Y,Z$ such that $X=Y\uplus Z$.
%Let $\F_Y=\F-\AA_\F(Y)$.
We have
%$$\AA_\F(X)=\AA_{\F}(Y)\uplus \AA_{\F_Y}(Z).$$
$$\mathlarger{\AA_\F(X)\ =\ \AA_{\F}(Y)\ \uplus\ \AA_{\F-\AA_\F(Y)}(Z)}.$$

%$$\AA_\F(X)=\AA_{\F}(Y)\cup \AA_{\F}(Z).$$
%
%$$\AA_\F(X)=\bigcup_{x\in X } \AA_{\F}(x).$$
%\red{essai matroide, attention ajouter lemme avant sur Flat...}
%
%\red{In matroid terms, assuming $\F=\F_M(B)$, we have that $E\s \AA(Y)$ is a flat,
%and $\F_Y=\F_{M(E\s \AA(Y))}(B\s \AA(Y))$.
% }
%$$\AA_\F(X)=\AA_{\F- \AA_\F(X\s x)}(x)\cup\AA_\F(X\s x).$$
\end{lemma}

\begin{proof}
Let us denote $\F_Y=\F-\AA_\F(Y)$.
Assume  $E=e_1<\dots<e_n$. We prove the result by induction. We assume that 
$\AA_\F(X)\cap\{e_1,\dots,e_{i-1}\}=(\AA_{\F}(Y)\cup \AA_{\F_Y}(Z))\cap\{e_1,\dots,e_{i-1}\}$. 
And we apply the definition given in Lemma \ref{lem:act-closure-algo}.
If $e_i\in Y$ then $e_i\in \AA(X)$ and $e_i\in \AA(Y)$.
If $e_i\in Z$ then $e_i\in \AA(X)$ and $e_i\in \AA_{\F_Y}(Z)$.
If $\emptyset\subset C^*(B;e_i)^<\subseteq \AA(X)$, then $e_i\in \AA(X)$. Moreover, in this case, we have $\emptyset\subset C^*(B;e_i)^<\subseteq \AA_{\F}(Y)\cup \AA_{\F_Y}(Z))$ by induction hypothesis, then: either $\emptyset\subset C^*(B;e_i)^<\subseteq \AA_{\F}(Y)$, and in this case $e_i\in \AA(Y)$;  or $\emptyset\subset C^*(B;e_i)^<\s \AA_{\F}(Y)\subseteq  \AA_{\F_Y}(Z))$, and in this case $e_i\in \AA_{\F_Y}(Z)$.
%If all $c\in C^*(B;e_i)$ with $c<e_i$ satisfies $c\in\AA(X)$, then $e_i\in\AA(X)$.\par
If $e_i\not\in B$ and there exists $c\in C(B;e_i)$ with $c<e_i$ and $c\in\AA(X)$ then $e_i\in\AA(X)$. Moreover, in this case, by induction hypothesis, we have: either there exists $c\in C(B;e_i)$ with $c<e_i$ and  $c\in\AA(Y)$, and in this case $e_i\in \AA(Y)$; or there exist no $c\in C(B;e_i)$ with $c<e_i$ and  $c\in\AA(Y)$, and then there exists  $c\in C(B;e_i)\s \AA(Y)$ with $c<e_i$ and $c\in\AA_{\F_Y}(Z))$, and in this case $e_i\in \AA_{\F_Y}(Z)$.
In every other case, $e_i\not\in \AA(X)$, $e_i\not\in \AA(Y)$ and $e_i\not\in \AA_{\F_Y}(Z))$.
Finally, we have shown that, in every case, $e_i\in \AA(X)$ if and only if
$e_i\not\in \AA(Y)$ or $e_i\not\in \AA_{\F_Y}(Z))$,
which achieves the proof by induction.
Observe that the resulting union is a disjoint union since 
$\AA_{\F_Y}(Z))\cap \AA_\F'Y)=\emptyset$ by definition of $\F_y$.
\end{proof}

Proposition \ref{prop:dec_graph} (in terms of bipartite graphs/tableaux) and Proposition \ref{prop:dec_base} (the same result rephrased more specifically in terms of matroids)
below provide a general characterization of the active closure in terms of activities of fundamental graphs induced in minors. Hence it could be used to provide various decompositions of activities for (fundamental) bipartite graphs/tableaux. 
In what follows, it will be practically used  in a  restricted form, essentially when $X$ is the set of internally active elements greater than a given one.
%\red{a verfiier}
%In what follows, it will be practically used  in a  restricted form, essentially when $X$ is the greatest internally active element (and when $X$ is the set of all internally active elements).
%\red{a verfiier}

%\red{ou bien mettre ces lemmes en Proposition ?!!!}

\begin{prop} %[Decomposition Lemma of a Bipartite Graph] %{68}
\label{prop:dec_graph}
Let $E$ be a linearly ordered set.
Let $\F$ be a bipartite graph/tableau on $(B,E\s B)$.
Let $X\subseteq \Int(\F)$.
%With the previous notations, let $X\subseteq \Int(\F)$.\par
The set $\AA(X)$ is the unique subset $A$ of $E$ satisfying the following properties:
\begin{enumerate}[(i)]
\partopsep=0mm \topsep=0mm %\parsep=0mm %\itemsep=0mm
%\item for all $e\in E\s (A \cup B)$, we have $C(B;e)\subseteq E\s A$; 
\item \label{it1} for all $e\in E\s (A \cup B)$, we have $C(B;e) \cap A=\emptyset$; 

or equivalently: for all $b\in B\cap A$, we have $C^*(B;b)\subseteq A$;
\item \label{it2} $\Int(\F-A)=\Int(\F)\s X$;
\item \label{it3} $\Ext(\F-A)=\Ext(\F)$;
\item \label{it4} $\Int(\F-E\s A)=X$;
\item \label{it5} $\Ext(\F-E\s A)=\emptyset$.
\end{enumerate}
\end{prop}

%\purple{A FAIRE : ENONCE DUAL, et ENONCE MIXTE, doit ressembler... et on devrait pouvoir faire operqtions separement sur \Int et \Ext independamment !!!}

%\red{IMPORTANT POUR ABG2 : voir quelles conditions suffisent a determiner ! le preciser ?}

%\com{**?ventuellement mettre la d?finition de $\AA$ dans le lemme pr?c?dent***}

\begin{proof}
First we verify that the two properties stated in (\ref{it1}) are equivalent.
Indeed the first  property 
%for all $e\not\in B$, $e\not\in A\Rightarrow C(B;e)\subseteq E\s A$ 
can be written equivalently:
for all $e\in E\s B$, if $C(B;e)\cap A\not=\emptyset$ then $e\in A$; 
that is: 
for all $e\in E\s B$,  if $e\in C^*(B;b)$ for some $b\in A\cap B$ then $e\in A$;
that is: 
for all $e\in E\s B$, for all $b\in B$, if $e\in C^*(B;b)$ and $b\in A$ then $e\in A$;
that is: 
for all $b\in B$, if $b\in A$ then $C^*(B;b)\subseteq A$.

Next, we show that $\AA(X)$  satisfies the five properties (\ref{it1})-(\ref{it5}).
\partopsep=0mm \topsep=0mm \parsep=0mm \itemsep=0mm
\begin{enumerate}[(i)]
\item
Let $e\not\in \AA(X)\cup B$. If $C(B;e)=\{e\}$ then $C(B;e) \cap \AA(X)=\emptyset$. Otherwise, let $b\in B\cap C(B;e)$, which implies $e\in C^*(B;b)$.
If $b\in \AA(X)$ then $C^*(B;b)\subseteq \AA(X)$ by definition of $\AA(X)$, so $e\in \AA(X)$, which is a contradiction. Hence $C(B;e) \cap \AA(X)=\emptyset$.\par

\item 
%Let $b\in B\s \AA(X)$. We assume that
Let $b\in \Int(\F-\AA(X))$.
By definition of $\Int$, we have $b\in B\s \AA(X)$ and  $b=\min (C^*(B;b)\s \AA(X))$.
Then $C^*(B;b)^<\subseteq \AA(X)$, so $C^*(B;b)^<=\emptyset$
otherwise $b\in \AA(X)$ by definition of $\AA(X)$, so $b=\min (C^*(B;b))$, that is $b\in \Int(\F)$. Since $b\not\in\AA(X)$, we get $b\in \Int(\F)\s X$. Conversely, let $b\in \Int(\F)\s X$.
Since $b\not\in X$ and $b=\min(C^*(B;b))$ then
$b\not\in \AA(X)$, by definition of $\AA(X)$.
So $b=\min(C^*(B;b)\s \AA(X))$, that is $b\in \Int(\F-\AA(X))$.\par

\item 
%Let $e\in (E\s B)\s \AA(X)$. 
%We assume that 
Let $e\in \Ext(\F-\AA(X))$.
By definition of $\Ext$, we have $e\in (E\s B)\s \AA(X)$
and $e=\min (C(B;e)\s \AA(X))$.
Then there exist no $c\in C(B;e)\cap \AA(X)$ with $c<e$, otherwise $e\in \AA(X)$ (by Lemma \ref{lem:act-closure-algo}).
So $e=\min (C(B);e)$, that is $e\in \Ext(\F)$. Conversely,
let $e\in \Ext(\F)$. We have $e\not\in B$.
Since $e=\min C(B;e)$, we have $e\not\in \AA(X)$ (by Lemma \ref{lem:act-closure-algo}). So $e=\min ( C(B;e)\s \AA(X))$, that is $e\in \Ext(\F-\AA(X))$.

\item Let $b\in X$. Since $X\subseteq \Int(\F)$ and $X\subseteq \AA(X)$, $b$ is  internally active in $\F$
and $b\in \AA(X)$. So $b=\min (C^*(B;b))=\min (C^*(B;b)\cap \AA(X))$, so $b\in \Int(\F-E\s \AA(X))$. 
Conversely, let $b\in \Int(\F-E\s \AA(X))$.
Then $b=\min (C^*(B;b)\cap \AA(X))$ by definition of $\Int$. 
So $C^*(B;b)^< \not\subseteq \AA(X)$.
By Lemma \ref{lem:act-closure-algo}, $b\in\AA(X)$ implies $b\in X$ or $C^*(B;b)^< \subseteq \AA(X)$. So we have $b\in X$.

\item Assume $e\in \Ext(\F-E\s \AA(X))$.
We have $e\in \AA(X)$, so, by  Lemma \ref{lem:act-closure-algo}, there exists $c<e$ in $C(B;e)\cap \AA(X)$,  so $e\not=\min (C(B;e)\cap \AA(X))$, a contradiction with the definition of $\Ext$. Hence $\Ext(\F-E\s A)=\emptyset$.\par
\end{enumerate}

Now, let $A\subseteq E$ satisfying these five properties.
We show that $A$ satisfies Definition \ref{def:dec_operation} of $\AA(X)$.
The property (\ref{it4}) implies $X\subseteq A$, which is the first property to satisfy in Definition \ref{def:dec_operation}. 
%The property (\ref{it1}) 
%%for all $e\not\in B$, $e\not\in A\Rightarrow C(B;e)\subseteq E\s A$ 
%can be written equivalently:
%for all $e\in E\s B$, if $C(B;e)\cap A\not=\emptyset$ then $e\in A$; 
%that is: 
%for all $e\in E\s B$,  if $e\in C^*(B;b)$ for some $b\in A\cap B$ then $e\in A$;
%that is: 
%for all $e\in E\s B$, for all $b\in B$, if $e\in C^*(B;b)$ and $b\in A$ then $e\in A$;
%that is: 
%for all $b\in B$, if $b\in A$ then $C^*(B;b)\subseteq A$.
As shown above, the property (\ref{it1}) can be stated: for all $b\in B$, if $b\in A$ then $C^*(B;b)\subseteq A$, which is the second property to satisfy in Definition \ref{def:dec_operation}. 
Finally, assume that there exists $b\in B$ such that $\emptyset \subset C^*(B;b)^<\subseteq A$ and $b\not\in A$. Then $b\not\in \Int(\F)$ as $C^*(B;b)^<\not=\emptyset$.
And $b=\min( C^*(B;b)\s A)$ as $C^*(B;e)^<\subseteq A$, that is: $b\in \Int(\F-A)$.
So $b\in \Int(\F-A)\s \Int(\F)$ which is a contradiction with property (\ref{it2}).
So $A$ satisfies the third property in Definition \ref{def:dec_operation}. 
Since $A$ satisfies the three properties in Definition \ref{def:dec_operation}, and $\AA(X)$ is the smallest set satisfying those three properties, we have shown $\AA(X)\subseteq A$.
%Hence $\AA(X)\subseteq A$ by Property \ref{prop:dec_operation} (i).\par

To conclude, let us assume that there exists $e\in A\s \AA(X)$. 
In a first case, let us assume that $e\in B$.
Then $C^*(B;e)\subseteq A$ by property (\ref{it1}). 
If $e=\min (C^*(B;e))$ then we have $e\in \Int(\F-E\s A)$ (by definition of $\Int$, since $e\in A$), which implies $e\in X$ by property (\ref{it4}), which is a contradiction with $e\not\in \AA(X)$.
%$e\in \Int(\F-E\s A)\s X$ which is a contradiction with property (iv), 
%either
So there exists $f<e$ in $C^*(B;e)\s\AA(X)$ (otherwise $\emptyset \subset C^*(B;e)^<\subseteq \AA(X)$, which implies $e\in\AA(X)$ by definition of $\AA(X)$). So there exists $f<e$ with $f\in A\s \AA(X)$.
In a second case, let us assume that $e\not\in B$.
Then, by property (\ref{it5}), there exists $f<e$ with $f\in C(B;e)\cap A$ (otherwise
$e$ is externally active in $\F-E\s A$). 
By assumption we have $e\in E\s(\AA(X)\cup B)$,
so, by property (\ref{it1}) satisfied by $\AA(X)$,
 we have
$C(B;e)\subseteq E\s \AA(X)$. So we have $f\in A\s \AA(X)$.
In any case, the existence of $e$ in $A\s \AA(X)$ implies the
existence of $f<e$ in $A\s \AA(X)$, which is impossible. 
So we have proved $A=\AA(X)$.
\emenew{apparemment pty (iii) non utilisee....}%
\end{proof}

%\red{************ NOW WE TURN TO MATROIDS *************}

\begin{prop}[equivalent to Proposition \ref{prop:dec_graph}] %[Decomposition Lemma] %{70}
\label{prop:dec_base}
%\red{a mettre en proposition?}
%\red{!!!!!! OUGJENSUIS !!!!!!!!!}
Let $E$ be a linearly ordered set.
Let $B$ be a basis of a matroid $M$ on $E$ with fundamental graph $\F$. Let $X\subseteq \Int_M(B)$.
The set $F=E\s\AA(X)$ is the unique subset of $E$
satisfying the following properties:
%\red{DOIT ON METTRE FLAT DANS HYPOTHESE ???}
\begin{enumerate}[(i)]
\item $B\cap F$ is a basis of $M(F)$, and, equivalently, $B\s F$ is a basis of $M/F$\par
\item  $\Int_{M(F)}(B\cap F)=\Int_M(B)\s X$,\par 
\item  $\Ext_{M(F)}(B\cap F)=\Ext_M(B)$,\par
\item  $\Int_{M/F}(B\s F)=X$,\par
\item  $\Ext_{M/F}(B\s F)=\emptyset$.
\end{enumerate}
\end{prop}

\begin{proof}
%\red{A REPRENDRE}
This proposition is essentially a reformulation of 
Proposition \ref{prop:dec_graph}
%the Decomposition Lemma for Bipartite Graphs \ref{prop:dec_graph}
in the language of matroids, using Property \ref{pty:fund_graph_flat}.
Let $F=E\s\AA(X)$.
By Proposition \ref{prop:dec_graph}, $E\s F$ is the unique subset of $E$ satisfying  properties (\ref{it1})-(\ref{it5}) stated in Proposition \ref{prop:dec_graph}.
Observe that property (\ref{it1}) is stated as:
for all $e\in E\s ((E\s F) \cup B)$, we have $C(B;e) \cap (E\s F)=\emptyset$.
That is, equivalently:
for all $e\in F\s B$, we have $C(B;e) \subseteq F$.
That is, equivalently, by Property \ref{pty:fund_graph_flat}: $B\cap F$ is a basis of $M(F)$.
Now, by  Property \ref{pty:fund_graph_flat}, properties (\ref{it2})-(\ref{it5}) of Proposition \ref{prop:dec_graph} translate directly to 
properties (\ref{it2})-(\ref{it5}) of the present result.
%\red{raccorucir grace a pty}
%
%\item \label{it2} $\Int(\F-(E\s F))=\Int(\F)\s X$; that is: $\Int_{M(F)}(B\cap F)=\Int_M(B)\s X$.
%\item \label{it3} $\Ext(\F-A)=\Ext(\F)$;
%\item \label{it4} $\Int(\F-E\s A)=X$;
%\item \label{it5} $\Ext(\F-E\s A)=\emptyset$.
%\end{enumerate}
%
%A flat $F$ satisfies (i) if and only if for all $e\in F\s B$ we have $C(B;e)\subseteq F$,
%which is property (i) of Proposition \ref{prop:dec_graph} with $A=E\s F$ and $\F=\F_M(B)$.
%With Property \ref{pty:fund_graph_flat} which assume $F$ satisfies property (i), 
%the properties (i), (ii), (iii), (iv), and (v) are equivalent to properties 
%(i), (ii), (iii), (iv), and (v) of Proposition \ref{prop:dec_graph} with $A=E\s F$ and $\F=\F_M(B)$.
%So $F=E\s \AA(X)$ is necessary to have these properties.
%Conversely $\AA(X)=\cup_{b\in \AA(X)}C^*(B;b)$ (immediate consequence of the definition)
%so $E\s \AA(X)$ is a flat of $M$ (union of cocircuits) and $F=E\s \AA(X)$ is sufficient. 
\end{proof}

\begin{prop} %[reformulation of \cite{EtLV98} \red{NON}]
\label{prop:dec-base-cyclic-flat}
Let $E$ be a linearly ordered set.
Let $\F$ be a bipartite graph/tableau on $(B,E\s B)$ (or equivalently let $B$ be a basis of a matroid $M$ on $E$ with fundamental graph $\F$).
We have
%$$\AA\Bigl(\Int(\F)\Bigr)=E\s \AA\Bigl(\Ext(\F)\Bigr).$$
$$E \ = \ \AA\Bigl(\Int(\F)\Bigr) \ \uplus \ \AA\Bigl(\Ext(\F)\Bigr).$$
%Moreover, assume $B$ is a basis of a matroid $M$ on $E$ with fundamental graph $\F$.
%Let $F=\AA(\Int(\F))$. We have that $F$ is the unique cyclic flat of $M$ such that:
%\begin{enumerate}[(i)]
%\item $B\cap F$ is a basis of $M(F)$,\par
%\item $B\s F$ is a basis of $M/F$,\par
%\item  $\Int_{M(F)}(B\cap F)=\emptyset$,\par 
%\item  $\Ext_{M(F)}(B\cap F)=\Ext_M(B)$,\par
%\item  $\Int_{M/F}(B\s F)=\Int_M(B)$,\par
%\item  $\Ext_{M/F}(B\s F)=\emptyset$.
%\end{enumerate}
%\red{integrer EtLV98 et algo}
\end{prop}

\begin{proof}
By Proposition \ref{prop:dec_graph}, $\AA(\Int(\F))$ is the unique subset $A\subseteq E$ such that:
%denoting $A=\AA(\Int(\F))$, 
%we have:
\vspace{-2mm}
%is the unique subset $A$ of $E$ such that:
\begin{itemize}[-]
\partopsep=0mm \topsep=0mm \parsep=0mm \itemsep=-1mm
%\item for all $e\in E\s (A \cup B)$, we have $C(B;e)\subseteq E\s A$; 
%\item for all $e\in E\s (A \cup B)$, we have $C_\F(B;e) \cap A=\emptyset$;
\item  for all $b\in B\cap A$, we have $C^*(B;b)\subseteq A$;
\item $\Int(\F-A)=\emptyset$;
\item $\Ext(\F-A)=\Ext(\F)$;
\item $\Int(\F-E\s A)=\Int(\F)$;
\item $\Ext(\F-E\s A)=\emptyset$.
\end{itemize}
\vspace{-2mm}
%Observe that 
Now we apply Proposition \ref{prop:dec_graph} to $\F^*$, bipartite graph on $(E\s B,B)$, with $X=\Ext(\F)=\Int(\F^*)$. We get that $E\s \AA(\Ext(\F))$ is the unique subset $E\s A'$ of $E$ such that the following properties hold, where we replace the statements of properties of $\F^*$ with equivalent statements for $\F$:
\vspace{-2mm}
\begin{itemize}[-]
\partopsep=0mm \topsep=-1mm \parsep=0mm \itemsep=-1mm
%\item for all $e\in E\s (A \cup B)$, we have $C(B;e)\subseteq E\s A$; 
\item for all $e\in E\s ((E\s A') \cup (E\s B))$, we have $C_{\F^*}(B;e) \cap (E\s A')=\emptyset$;
\item that is equivalently: for all $e\in A'\cap B$, we have $C^*_\F(B;e) \subseteq A'$;
%\red{ATTENTION BIZARRE, cette propriete est exactement la meme que pour $\F$, avant dualisation !!!!! A VERIFIER !!!}
\item $\Ext(\F-E\s A')=\emptyset$;
\item $\Int(\F-E\s A')=\Int(\F)$;
\item $\Ext(\F-A')=\Ext(\F)$;
\item $\Int(\F-A')=\emptyset$.
\end{itemize}
\vspace{-2mm}
Finally, the properties satisfied by $A$ and by $A'$ are exactly the same, hence $A=A'$ by uniqueness in Proposition \ref{prop:dec_graph}, that is $\AA(\Int(\F))=E\s \AA(\Ext(\F))$.
%Observe that if $e\in \AA(\Int(\F))\cap B$ then, by definition of $\AA$, we have $C^*_\F(B;e)\subseteq \AA(\Int(\F))$, that is  $C^*_\F(B;e) \s \AA(\Int(\F))=\emptyset$.
%So $\AA(\Int(\F))$ satisfies all the above properties of $A'$, hence it is equal to $A'$, by uniqueness in Proposition \ref{prop:dec_graph}.
%We have proved $\AA(\Int(\F))=E\s \AA(\Ext(\F))$.
%
%\red{completer le reste}
\end{proof}

%
%\begin{color}{gray}
%\begin{cor}[reformulation of \cite{EtLV98}]
%\red{ATTENTION COROLLAIRE DU LEMME PAS DE LA PROP!!!}
%Let $E$ be a linearly ordered set.
%Assume $B$ is a basis of a matroid $M$ on $E$ with fundamental graph $\F$.
%Let $F=\AA(\Int(\F))$. We have that $\AA(\Int(\F))$ is the unique cyclic flat $F$ of $M$ such that:
%\begin{enumerate}[(i)]
%\item $B\cap F$ is a basis of $M(F)$,\par
%\item $B\s F$ is a basis of $M/F$,\par
%\item  $\Int_{M(F)}(B\cap F)=\emptyset$,\par 
%\item  $\Ext_{M(F)}(B\cap F)=\Ext_M(B)$,\par
%\item  $\Int_{M/F}(B\s F)=\Int_M(B)$,\par
%\item  $\Ext_{M/F}(B\s F)=\emptyset$.
%\end{enumerate}
%%\red{integrer EtLV98 et algo}
%\end{cor}
%\end{color}

\begin{definition}
\label{def:dec-cyclic-flat}
Let $E$ be a linearly ordered set.
Let $\F$ be a bipartite graph/tableau on $(B,E\s B)$ (or equivalently let $B$ be a basis of a matroid $M$ on $E$ with fundamental graph $\F$).
The set $F=\AA(\Ext(\F))$ is called \emph{the external part of $E$  w.r.t. $\F$}, and the set $E\s F=\AA(\Int(\F))$ is called \emph{the internal part of $E$ w.r.t. $\F$}. 
Observe that, in the case where $B$ is a basis of a matroid $M$, $F_c$ is  a cyclic flat of $M$ (as $\AA(\Ext(\F))$, resp.  $\AA(\Int(\F))$, is a union of circuits, resp. cocircuits).
\end{definition}

From Proposition \ref{prop:dec-base-cyclic-flat}, using the formulation used in Proposition \ref{prop:dec_base}, we directly retrieve the following result from \cite{EtLV98} (in an equivalent form).
Let us mention that we complete it with a practical characterization in Corollary \ref{cor::dec-base-cyclic-flat-algo} below. 

%\begin{definition}[\cite{EtLV98}]
\begin{cor}[\cite{EtLV98}]
\label{cor:dec-cyclic-flat}
%\red{attention: definir extrnal part aussi pour $\F$}
Let $B$ be a basis of a matroid $M$ on a linearly ordered set $E$ with fundamental graph $\F$. Let $F_c$ be the external part of $E$ w.r.t. $\F$.
%(the fundamental graph $\F$ of) $B$.
%The set $F=\AA(\Ext(\F))$ is called \emph{the external part of $E$  w.r.t. $B$}, and the set $E\s F=\AA(\Int(\F))$ is called \emph{the internal part of $E$ w.r.t. $B$}. 
%Observe that $Fc$ is  a cyclic flat of $M$ (as $\AA(\Ext(\F))$, resp.  $\AA(\Int(\F))$, is a union of circuits, resp. cocircuits).
The subset $F_c$ is the unique subset (or cyclic flat) $F$ of $M$ such that:
\begin{enumerate}[(i)]
\item $B\cap F$ is a basis of $M(F)$, and  $B\s F$ is a basis of $M/F$,\par
\item  $\Int_{M(F)}(B\cap F)=\emptyset$,\par 
\item  $\Ext_{M(F)}(B\cap F)=\Ext_M(B)$,\par
\item  $\Int_{M/F}(B\s F)=\Int_M(B)$,\par
\item  $\Ext_{M/F}(B\s F)=\emptyset$.\hfill\square
\end{enumerate}
\end{cor}
%\end{definition}

\begin{cor}
\label{cor::dec-base-cyclic-flat-algo}
%{a mettre apres TH / Corollaire finaux ?}
Let $E$ be a linearly ordered set.
%\blue{Let $E=e_1<...<e_n$ be a linearly ordered set.}
%\red{aussi pour $\F$}
%Let $B$ be a basis of a matroid $M$ on $E$ with fundamental graph $\F$.
Let $\F$ be a bipartite graph/tableau on $(B,E\s B)$ (or equivalently let $B$ be a basis of a matroid $M$ on $E$ with fundamental graph $\F$).
The partition of $E$ into internal and external parts w.r.t. $\F$ is given by the following definition (yielding a linear algorithm by a single pass over $E$ in increasing order).
%\red{A REPRENDRE, $e$ ou $e_i$ ???}

\begin{algorithme} 
%\blue{For $i$ from $1$ to $n$ do}\par
\hskip 1cm 
If $e\in B$: if there exists $c<e$ external in $C^*(B;e)$ then $e$ is external\par
\hskip 1cm 
$\hphantom{\hbox{for $e\in B$:}}$ otherwise $e$ is internal\par
\hskip 1cm
If $e\not\in B$: if there exists $c<e$ internal in $C(B;e)$ then $e$ is internal,\par
\hskip 1cm $\hphantom{\hbox{for $e\not\in B$:}}$ otherwise $e$ is external\par
\end{algorithme}
\end{cor}

\begin{proof}
Observe that if $e$ is internally, resp. externally, active then $C^*(B;e)^<=\emptyset$, resp. $C(B;e)^<=\emptyset$, and then $e$ is internal, resp. external.
Then, the computation of the internal part comes directly from Lemma \ref{lem:act-closure-algo} applied to $X=\Int(\F)$. 
% and Lemma \ref{lem:act-closure-algo-bas} 
The other cases, where $e$ is not internal, imply that $e$ is external, equivalently either by duality (the cases are dual), or by Proposition \ref{prop:dec-base-cyclic-flat}.
\end{proof}

\begin{lemma}
\label{lem:closure-in-internal-part}
Let $E$ be a linearly ordered set.
Let $\F$ be a bipartite graph/tableau on $(B,E\s B)$ (or equivalently let $B$ be a basis of a matroid $M$ on $E$ with fundamental graph $\F$).
Let $X\subseteq \Int(\F)$. Let $F_c$ be the external part of $E$ w.r.t. $\F$.
We have
$$\mathlarger{\AA_\F(X)\ =\ \AA_{\F-F_c}(X)}.$$
Moreover, the external part of $E\s \AA_\F(X)$ w.r.t. $\F-\AA_\F(X)$ is also $F_c$.
\end{lemma}

\begin{proof}
We have $\AA(\Int(\F))\cap \AA(\Ext(\F))=\emptyset$ (Proposition \ref{prop:dec-base-cyclic-flat}), hence $\AA(X)\cap F_c=\emptyset$.
Then, first, one sees directly that the computation of $\AA(X)$ given by Lemma \ref{lem:act-closure-algo-bas} yields the same result as if it is applied to $\F-F_c$.
So $\AA_\F(X)\ =\ \AA_{\F-F_c}(X)$.
And, second, for the same reason, 
the computation of $\AA(\Ext(\F))=\AA(\Int(\F^*))$ given by Lemma \ref{lem:act-closure-algo-bas} applied to $\F^*$ yields the same result as if it is applied to $\F^*-\AA(X)$.
So $\AA_\F(\Ext(\F))\ =\ \AA_{\F-\AA(X)}(\Ext(\F-\AA(X)))$.
\end{proof}

\begin{definition}
\label{def:act-seq-dec-fund-graph}
Let $E$ be a linearly ordered set.
Let $\F$ be a bipartite graph/tableau on $(B,E\s B)$, 
(or equivalently let $B$ be a basis of a matroid $M$ on $E$ with fundamental graph $\F$),
with $\io$ internally active elements $a_1<...<a_\io$ and $\ep$ externally active elements $a'_1<...<a'_\ep$.
The \emph{active filtration of $\F$} (or $B$) is the sequence of subsets $(F'_\ep, \ldots, F'_0, F_c , F_0, \ldots, F_\io)$ of $E$ %\red{*ecrire inclusions? oui!*}
defined by the following:
$$F_c=\AA(\Ext(\F))=E\s \AA(\Int(\F));$$
$F_\io=E$, and
for every $0\leq k\leq\io-1$, 
$$F_k=E\setminus \AA(\{a_{k+1},\dots, a_\io\});$$
$F'_\ep=\emptyset$,  and for every $0\leq k\leq\ep-1$, 
$$F'_k=\AA(\{a'_{k+1},\dots, a'_\ep\}).$$
\end{definition}

%\red{BIEN VERIFIER INDICES}
%
%\red{NB on peut def aussi par $F_{k-1}=F_k\s ...$ cf these ou autre prop, d'apres lemme inductif}
%
%
%\red{*****dessous organisation de orientation transposee**********}
%

\begin{lemma}
\label{lem:pty-act-dec-seq-bas}
%\red{dessous a mettre en lemme ou en observation car tres utile et important}
Using the above notations,
 we have $$\emptyset= F'_\ep\subset...\subset F'_0=F_c=F_0\subset...\subset F_\io= E.$$
The active filtration of $\F$ is a filtration of $E$ (Definition~\ref{def:general-filtration}).
Moreover, we have, for $1\leq k\leq\io$, 
\begin{align*}
&& F_k\setminus F_{k-1}&= \AA(\{a_k,\dots,a_\io\})\s \AA(\{a_{k+1},\dots,a_\io\})\\
&&&=\AA_{\F-(E\s F_k)}(\{a_k\}),\\
&&\min(F_k\setminus F_{k-1})&=a_k,
\end{align*}
and, for $1\leq k\leq \ep$,  
\begin{align*}
&&F'_{k-1}\setminus F'_{k}&=
 \AA(\{a'_k,\dots,a'_\ep\})\s \AA(\{a'_{k+1},\dots,a_\ep\})\\
&&&=\AA_{\F-F'_{k}}(\{a'_k\}),\\
&&\min(F'_{k-1}\setminus F'_{k})&=a'_k.
\end{align*}
%\red{ajouter matroid setting}
%
%Moreover, for $0\leq k\leq\io$, we have that $F_k$ satisfies:
% for all $b\in B\s F_k$, $C^*_M(B;b)\cap F_k=\emptyset$.
%And, for $0\leq k\leq\ep$, we have that $F'_k$ satisfies:
% for all $b\in B\s F'_k$, we have $C^*_M(B;b)\cap F'_k=\emptyset$.
%
Moreover, in the case where $B$ is a basis of a matroid $M$,  we have:
\begin{itemize}
\partopsep=0mm \topsep=0mm \parsep=0mm \itemsep=0mm
\item for $0\leq k\leq\io$, $F_k$ satisfies the properties of Property \ref{pty:fund_graph_flat}, and $F_k$ is a flat of $M$;
\item for $0\leq k\leq\ep$, $F'_k$ satisfies the properties of Property \ref{pty:fund_graph_flat}, and $F'_k$ is a dual-flat of $M$.
\end{itemize}
In particular, $F_0=F_c=F'_0$ is a cyclic-flat of $M$.

%\red{rk que $F_k$ et $F'_k$ saisfont  Property \ref{pty:fund_graph_flat} ? c'est vu dans prop suivante... et attention Pty donne pour matroide... je l'ai ajoute dans Rk avant observation}
%where the second equalities come directly from Lemma \ref{lem:act-closure-part} (applied to $\F^*$ in the second case).
\end{lemma}

\begin{proof}
Since $\AA$ is increasing, for $1\leq k\leq\io$, we have $\AA(\{a_{k+1},\dots,a_\io\})\subseteq \AA(\{a_{k},\dots,a_\io\})\subseteq \AA(\Int(\F))$, and, by definition of $\AA$,  $a_k\in \AA(\{a_{k},\dots,a_\io\})\s \AA(\{a_{k+1},\dots,a_\io\})$.
So
$F_c\subseteq F_{k-1}\subset F_k$. 
And dually, we have, for $1\leq k\leq\ep$, $F'_{k}\subset F_{k-1}\subseteq F_c$.
So we have $\emptyset= F'_\ep\subset...\subset F'_0=F_c=F_0\subset...\subset F_\io= E$.

 Moreover $a_k=\min(\AA(\{a_{k},\dots,a_\io\})$ so  $a_k=\min(F_k\setminus F_{k-1})$, $1\leq k\leq\io$, which is increasing with $k$ by hypothesis.
And, dually, we have $a'_k=\min(F'_{k-1}\setminus F'_k)$, $1\leq k\leq\ep$, which is increasing with $k$ by hypothesis.
So the active filtration is a filtration of $E$, according to Definition \ref{def:general-filtration}.

Moreover, by Lemma \ref{lem:act-closure-part}, we have 
$ \AA(\{a_k,\dots,a_\io\})= \AA(\{a_{k+1},\dots,a_\io\})\uplus \AA_{\F-(E\s F_k)}(a_k)$.
And by Lemma \ref{lem:act-closure-part} applied to $\F^*$, we have 
$ \AA(\{a'_k,\dots,a'_\ep\})= \AA(\{a'_{k+1},\dots,a'_\ep\})\uplus \AA_{\F-F'_k}(a'_k)$.

For $0\leq k\leq\io$, by Proposition \ref{prop:dec_graph},
we have that $F_k$ satisfies:
for all $b\in B\cap (E\s F_k)$, we have $C^*(B;b)\subseteq (E\s F_k)$.
That is: for all $b\in B\s F_k$, $C^*_M(B;b)\cap F_k=\emptyset$.
Hence $F_k$ satisfies properties of Property \ref{pty:fund_graph_flat}.
 And $F_k$ is a flat of $M$ as its complement is a union of cocircuits.
 
For $0\leq k\leq\ep$, by Proposition \ref{prop:dec_graph} applied to the dual $\F^*$ (or  the basis $E\s B$ of $M^*$),
we have that $F'_k$ satisfies:
for all $e\in (E\s B)\cap F'_k$, we have $C^*_{M^*}(E\s B;e)\subseteq F'_k$.
That is:
 for all $e\in F'_k\s B$, we have $C_M(B;e)\subseteq F'_k$.
 Hence $F'_k$ satisfies properties of Property \ref{pty:fund_graph_flat}.
 And $F'_k$ is a dual-flat of $M$ as it is a union of circuits.
 \end{proof}

\begin{definition}
\label{def:act-part-bas}
Using the above notations, the active filtration of $\F$ (or $B$), induces a partition of the ground set $E$, which we call \emph{the active partition of $\F$ (or $B$)}:
$$E= (F'_{\ep-1}\s F'_\ep)\ \uplus\ \dots \ \uplus\ (F'_{0}\s F'_1)\ \uplus\ (F_1\s F_0)\ \uplus\ \dots \ \uplus\ (F_{\io}\s F_{\io-1}).$$
%We call this partition \emph{the active partition of $\F$ (or $B$)}.
% and we assume that it is always given with the external part $F_c$, 
%%(which is a cyclic-flat in the case of a matroid), 
%i.e. it can be thought of as a pair of partitions, one for $F_c$, the other for $E\s F_c$.
%%\medskip
\emevder{je viens juste au dernier moment de definir active minors dans ce papier, terminoogie jamais utilisee dans le papier ! mais mis par coherence avec AB2-b}%

\noindent Also, we call \emph{active minors w.r.t. $\F$ (or $B$)} the minors induced by the active filtration of $\F$ (or $B$), that is
 the minors $\M(F_k)/F_{k-1}$ for $1\leq k\leq\io$,
and the minors
$\M(F'_{k-1})/F'_{k}$ for all $1\leq k\leq \ep$.
\end{definition}

%\red{rq: pas besoin de $F_c$ puisque external parts ont leur plus petit element hors de la base !!!}

\begin{observation}
The active partition of $\F$ (or $B$) 
determines the active filtration of $\F$ (or $B$), hence it is an equivalent notion.
Precisely, using the above notations, knowing only the subsets forming the active partition of $\F$ (or $B$) allows us to build:
\vspace{-1mm}
\begin{itemize}[-]
\partopsep=-1mm \topsep=-1mm \parsep=-1mm \itemsep=-1mm
\item the subset $F_c$ of the active filtration of $\F$ (or $B$), since the smallest element of a part is in $B$ if and only if this part is of type
$F_k\setminus F_{k-1}$ for some $1\leq k\leq\io$;
\item the active filtration of $\F$ (or $B$), since the sequence $\min(F_k\setminus F_{k-1})$, $1\leq k\leq\io$,  is increasing with $k$, and the sequence  $\min(F'_{k-1}\setminus F'_k)$, $1\leq k\leq\ep$, is increasing with $k$, so that the position of each part of the active partition with respect to the active filtration is identified.
\end{itemize}
\end{observation}
%
%\red{remarque a faire des section sur dec seq ??? si on y definit partition...}
%\red{INDICES ET DEF A VERIFIER + $\F_M$ NON DEFINI} 
%\red{BY LEMMA INDUCTIF POURSECONDE EGALITE}
%\begin{color}{teal}
%Also, we have, for $1\leq k\leq\io$, 
%$$F_k\setminus F_{k-1}= \AA(\{a_k,\dots,a_\io\})\s \AA(\{a_{k+1},\dots,a_\io\})=\AA_{\F_{M(F_k)}(B\cap F_k)}(a_k),$$
%and, for $1\leq k\leq \ep$,  $$F'_{k-1}\setminus F'_{k}=
% \AA(\{a'_k,\dots,a'_\ep\})\s \AA(\{a'_{k+1},\dots,a_\ep\})=\AA_{\F_{M^*(E\s F_{k-1})}(F'_k\s B)}(a'_k).$$
%\end{color}
%Also, we have, for $1\leq k\leq\io$, 
%\begin{align*}
% F_k\setminus F_{k-1}&= \AA(\{a_k,\dots,a_\io\})\s \AA(\{a_{k+1},\dots,a_\io\})&&\\
%%&&&=\AA_{\F_{M(F_k)}(B\cap F_k)}(a_k)&&\text{(by Lemma \ref{lem:act-closure-part})}
%&=\AA_{\F-(E\s F_k)}(a_k)&&\text{(by Lemma \ref{lem:act-closure-part}),}
%\end{align*}
%and, for $1\leq k\leq \ep$,  
%\begin{align*}
%F'_{k-1}\setminus F'_{k}&=
% \AA(\{a'_k,\dots,a'_\ep\})\s \AA(\{a'_{k+1},\dots,a_\ep\})&&\\
%%&&&=\AA_{\F_{M^*(E\s F_{k-1})}(F'_k\s B)}(a'_k)&& \text{(by Lemma \ref{lem:act-closure-part}),}.
%&=\AA_{\F-F'_{k-1}}(a'_k)&& \text{(by Lemma \ref{lem:act-closure-part} applied to $\F^*$).}
%\end{align*}

From a constructive viewpoint, let us remark that, by Lemma \ref{lem:pty-act-dec-seq-bas}, and more generally by Lemma \ref{lem:act-closure-part}, the active partition of $\F$ can be computed directly from $\F$, or also from the successive subgraphs of $\F$ induced by the active filtration of $\F$,
computing the active closure of active elements one by one
(or also from successive corresponding minors in a matroid setting, by Property \ref{pty:fund_graph_flat}, as made explicit in next Theorem~\ref{thm:unique-dec-seq-bas}).

Moreover, and more practically, Proposition \ref{prop:basori-partact} (postponed at the end of the paper) gives a direct construction of the active partition by a linear  algorithm consisting in a single pass over $E$.
\ss

%Formally, we state the next observation, that comes directly from  Lemma \ref{lem:act-closure-part} (applied to $\F$, and to $\F^*$ in the second case).
%\red{***}
%\begin{observation}
%%\red{dessous a mettre en lemme ou en observation car tres utile et important}
%Using the above notations,
% we have, for $1\leq k\leq\io$, 
%\begin{align*}
%&& F_k\setminus F_{k-1}&= \AA(\{a_k,\dots,a_\io\})\s \AA(\{a_{k+1},\dots,a_\io\})\\
%&&&=\AA_{\F-(E\s F_k)}(a_k)
%\end{align*}
%and, for $1\leq k\leq \ep$,  
%\begin{align*}
%&&F'_{k-1}\setminus F'_{k}&=
% \AA(\{a'_k,\dots,a'_\ep\})\s \AA(\{a'_{k+1},\dots,a_\ep\})\\
%&&&=\AA_{\F-F'_{k}}(a'_k).
%\end{align*}
%%where the second equalities come directly from Lemma \ref{lem:act-closure-part} (applied to $\F^*$ in the second case).
%\end{observation}

Finally, let us notice that, in the definitions that precede and the results that follow, the particular case of internal fundamental graphs (or internal bases) is addressed as the case where $F_c=\emptyset$, and case of external fundamental graphs (or external bases) is addressed as the case where $F_c=E$.
Those cases are dual to each other. Let us deepen this with the next observation,
which comes directly from Observation \ref{obs:dual-closure} (for duality), and from Lemma \ref{lem:closure-in-internal-part} (for restriction to $\F-(E\s F_c)$ or dually to $\F-F_c$).
%The next observation 
It will be deepened again in Observation \ref{lem:dec-seq-bas-observation-suite}.
%
%\begin{color}{teal}
%***AE***
%%\begin{lemma}
%\begin{observation}
%%\label{lem:dec-seq-bas-observation}
%%Using the same notations as in the above definitions,
%Using the above notations,
%let $\emptyset= F'_\ep\subset...\subset F'_0=F_c=F_0\subset...\subset F_\io= E$ be the active filtration of $\F$ (or $B$). 
%We have: %\red{IN THE MATROID CASE....}
%%\begin{itemize}
%\begin{enumerate}
%\item $\emptyset= E\s F_\io\subset...\subset E\s F_0=E\s F_c=E\s F'_0\subset...\subset E\s F'_\ep= E$ is the active filtration of $\F^*$, for the cyclic-flat $E\s F_c$ of $M^*$ in the matroid case;
%\label{item:sec-dec-act-dual-bas}
%\item $\emptyset= F'_\ep\subset...\subset F'_0=F_c=F_c$ is the active filtration of the external basis $B\cap F_c$ of $M(F_c)$, for the cyclic-flat $F_c$ of $M(F_c)$;
%\item $\emptyset=\emptyset= F_0\s F_c\subset...\subset F_\io\s F_c=E\s F_c$ is the active filtration of the internal basis $B\s F_c$ of  $M/F_c$, for the cyclic-flat $\emptyset$ of $M/F_c$.
%%\end{itemize}
%\end{enumerate}
%\end{observation}
%%\end{lemma}
%\end{color}

\begin{observation}
\label{lem:dec-seq-bas-observation}
Using the above notations,
let $\emptyset= F'_\ep\subset...\subset F'_0=F_c=F_0\subset...\subset F_\io= E$ be the active filtration of $\F$ (or $B$), with external part $F_c$. 
We have: %\red{IN THE MATROID CASE....}
%\begin{itemize}
\begin{enumerate}
\item 
\label{item:sec-dec-act-dual-bas}
$\emptyset= E\s F_\io\subset...\subset E\s F_0=E\s F_c=E\s F'_0\subset...\subset E\s F'_\ep= E$ is the active filtration of $\F^*$ (or of the basis $E\s B$ of $M^*$), with external part $E\s F_c$; % of $M^*$ in the matroid case;
\item 
\label{item:sec-dec-act-external-bas}
$\emptyset= F'_\ep\subset...\subset F'_0=F_c=F_c$ is the active filtration of $\F-(E\s F_c)$ (or of the external base $B\cap F_c$ of $M(F_c)$), with external part $F_c$; %for the cyclic-flat $F_c$ of $M(F_c)$;
\item 
\label{item:sec-dec-act-internal-bas}
$\emptyset=\emptyset= F_0\s F_c\subset...\subset F_\io\s F_c=E\s F_c$ is the active filtration of $\F-F_c$ (or of the internal base $B\s F_c$ of  $M/F_c$), with external part $\emptyset$. %for the cyclic-flat $\emptyset$ of $M/F_c$.
%\end{itemize}
\end{enumerate}
\end{observation}

For the sake of concision, we state the following Theorem \ref{thm:unique-dec-seq-bas} in terms of matroids (it is technically the main result of this section), but it could be equally stated in terms of bipartite graphs/tableaux as a decomposition into particular uniactive bipartite graphs/tableaux (using Property \ref{pty:fund_graph_flat} as previously for the translation). \emevder{We comment the uniqueness part of the theorem below  in Remark \ref{rk:important-uniqueness} below.}
%We leave details to the reader.
%\red{Sous-tableau tel que pty (i) du lemme de decomposition est satisfait, a preciser ?}

%\red{dessous theoreme ?}
%\begin{lemma}
\begin{thm}
\label{thm:unique-dec-seq-bas}
Let $E$ be a linearly ordered set.
%Let $\F$ be a bipartite graph/tableau on $(B,E\s B)$,
%resp. let $B$ be a basis of a matroid $M$ on $E$ with fundamental graph $\F$
Let $B$ be a basis of a matroid $M$ on $E$ with fundamental graph $\F$.
%\teal{ENLEVER CAR IMPLIQUE PAR DESSOUS : with $\io\geq 0$ internally active elements $a_1<...<a_\io$ and $\ep\geq 0$ externally active elements $a'_1<...<a'_\ep$.}
The active filtration of $\F$ 
%is a filtration of $M$, and it %, resp. $B$, 
%\red{peut eter pas mettre or $B$ chaque fois, definir pour $\F$ seulemenjt ???} 
is the unique (connected) 
%\red{ATTENTION ai mis (connected) a a place de (abstcrat) voir suite de enonce} 
filtration 
$\emptyset= F'_\ep\subset...\subset F'_0=F_c=F_0\subset...\subset F_\io= E$
%$(F'_\ep, \ldots, F'_0, F_c , F_0, \ldots, F_\io)$ 
of $E$ (or $M$)
%\red{NON !!!! on doit utiliser connxeite, pfaux pour abstrcat dec seq je pense} resp. 
%is the unique filtration $(F'_\ep, \ldots, F'_0, F_c , F_0, \ldots, F_\io)$ of $M$, 
such that:
%\red{ATTENTION, dessus, les tableaux doivent avoir la propriete que B induise une base !!! ce ne sont pas n'importels tableaux possibles... pourtant c'est juste prendre l'intersection, alors pourqoi pty en plus ??? BIZARRE !!! A VOIR !!! IDEE : pas de pb car ca depend du mineur ! mais pourtant on a pty en plus dans le lemme...}
\begin{itemize}
\item 
for all $1\leq k\leq\io$,
%the bipartite graph/tableau $\F_k=\F-(F_{k-1}\cup (E\s F_k))$, resp. 
the set $$B_k=B\cap F_k\s F_{k-1}$$ 
is a uniactive internal basis of the minor $$\M_k=\M(F_k)/F_{k-1};$$
%with unique internally active element $a_k=\min(F_k\setminus F_{k-1})$;
%\red{the set IS a basis}
\item
for all $1\leq k\leq \ep,$
% the bipartite graph/tableau $\F'_k=\F-(F'_{k}\cup (E\s F'_{k-1}))$, resp.
the set $$B'_k=B\cap F'_{k-1}\s F'_k$$ is a uniactive external basis of the minor $$\M'_k=\M(F'_{k-1})/F'_{k}.$$ 
% with unique externally active element $a'_k=\min(F'_{k-1}\setminus F'_{k})$.
%\red{***attention convetion pour matroide reuit a une seule ar?te}
\end{itemize}
%\end{lemma}
%In particular, in the case where $B$ be is a basis of a matroid $M$, the active filtration of $\F$ is a filtration of $M$ (see Definition \ref{def:matroid-dec-seq}).
Notice that the active filtration of $\F$ is actually a connected filtration of $M$ (Definition \ref{def:general-filtration}).
Notice also that, for $1\leq k\leq \io$, if $M(F_k)/F_{k-1}$ is an isthmus, then $B_k$ equals this isthmus, and that, for $1\leq k\leq \ep$, if $M(F'_{k-1})/F'_{k}$ is a loop, then $B'_k=\emptyset$.
\end{thm}

%\begin{proof}[Proof of Theorem \ref{thm:unique-dec-seq-bas}]
\begin{proof}
%\red{a verifier : on n'uilise pas Lemma \ref{lem:act-closure-part}, normal ?}
%
%\red{!!!!*******OUGJENSUIS*****!!!!!!}l
%
%Assume $\F$ has  $\io$ internally active elements $a_1<...<a_\io$, and $\ep$ externally active elements $a'_1<...<a'_\ep$.
First, let us directly check that the active filtration $(F'_\ep, \ldots, F'_0, F_c , F_0, \ldots, F_\io)$ satisfies the given properties.
The basis $B$ has $\io\geq 0$ internally active elements, which we denote $a_1<...<a_\io$, and $\ep\geq 0$ externally active elements, which we denote $a'_1<...<a'_\ep$.
% \red{(though it could be deduced also from the reasoning below proving uniqueness)}.
Let $1\leq k\leq\io$. %Let $a_k=\min(F_k\setminus F_{k-1})$.
By Lemma \ref{lem:pty-act-dec-seq-bas}, 
we have $a_k=\min(F_k\setminus F_{k-1})$ and 
we have $F_k\setminus F_{k-1}= \AA_{\F-(E\s F_k)}(\{a_k\})$.
Obviously, since $ \AA_{\F-(E\s F_k)}(\{a_k\}) \subseteq E\s F_{k-1}$, we have in fact
$F_k\setminus F_{k-1}=\AA_{\F-(E\s F_k)}(\{a_k\})=\AA_{\F-(F_{k-1}\cup (E\s F_k))}(\{a_k\})$.
By Property \ref{pty:fund_graph_flat}, we have that $B_k=B\cap F_k\s F_{k-1}$ is a basis of $M_k=M(F_k)/F_{k-1}$.
Let us denote $\F_k=\F-(F_{k-1}\cup (E\s F_k))=\F_{M_k}(B_k)$.
%\red{a observer precedemmnet  qu'on a bien une abse et que tableau induit est celui-ci ?}
%
%We have $F_k\setminus F_{k-1}=\AA_{\F_k}(\{a_k\})$.
%The bipartite graph/tableau $\F_k$ (or the matroid $M_k$) is defined on $F_k\setminus F_{k-1}$.
Since $F_k\setminus F_{k-1}= \AA_{\F_k}(\{a_k\})$, we have that $a_k$ is internally active in $\F_k$.
Moreover, by Lemma \ref{lem:closure-easy}, $F_k\setminus F_{k-1}= \AA_{\F_k}(\{a_k\})$ implies that $\F_k$ is uniactive internal.
%(and internal since $a_k\in \Int(\F)$).
Dually, let $1\leq k\leq \ep$.
%\red{utiliser lemme intuile... mais peut etre observer des la def que $\AA(X)\cap (\Int(\F)\cup \Ext(\F))=X$.}
By Lemma \ref{lem:pty-act-dec-seq-bas} and Property \ref{pty:fund_graph_flat},
we have similarly that $F'_{k-1}\setminus F'_{k}= \AA_{\F-F'_{k}}(a'_k)$,
that $a'_k=\min(F'_{k-1}\setminus F'_{k})$, that $B'_k=B\cap F_{k-1}\s F_{k}$ is a basis of $M'_k=M(F_{k-1})/F_{k}$, and that $\F'_k$ is uniactive external.
So, we have proved that the active filtration satisfies the given properties.
%\red{monter aussi la partie "is a basis"}.

Now, 
%in the matroid setting, that is in the case where $B$ is the basis of a matroid $M$, 
notice that each involved minor $M_k$, $1\leq k\leq\io$, or $M'_k$, $1\leq k\leq\io$, has a uniactive internal or a uniactive external basis, which implies that this minor is an isthmus (in this case the basis equals this isthmus), or a loop (in this case, the basis is the empty set), or a connected matroid (since $\beta(M)\not=0$ and $\mid E\mid>1$).
This proves that the active filtration of $\F$ is a connected filtration of $M$ (Definition \ref{def:general-filtration}).
%This also proves that those minors are connected as soon as they have more than one element, which achieves the proof that the active filtration of $\M$ is a filtration of $M$.

%\def\u{\underbar}
%\def\u{\tilde}
\def\S{{\mathcal S}}

It remains to prove the uniqueness property.
Assume that a filtration $\S=(F'_\ep, \ldots, F'_0,$ $F_c ,$  $F_0, \ldots, F_\io)$ satisfies the properties stated in the proposition.
Let us denote $a_k=\min(F_k\setminus F_{k-1})$, $1\leq k\leq \io$, and $a'_k=\min(F'_{k-1}\setminus F'_{k})$, $1\leq k\leq \ep$.

%We prove that it is equal to the active filtration 
%$\u\S=(\u F'_\ep, \ldots, \u F'_0, \u F_c , \u F_0, \ldots, {\u F}_\io)$ of $\F$.
%,  by induction over the number $\io+\ep$ of (internally or externally) active elements.
%If $\io+\ep=0$, then $\F$ is empty and the result is true.
%Assume the result is true for all bipartite graphs/tableaux with $\io+\ep-1$ active elements.

%Let us write the following proof in the matroid setting (equivalent to the bipartite graph/tableau setting as seen above).\red{a virer}

First, recall that in any matroid $M$, for every set $F$, the union of a basis of $M/F$ and a basis of $M(F)$ is basis of $M$.
Hence, since $B_k$ is a basis of $M_k$, $1\leq k\leq \io$, and
$B'_k$ is a basis of $M'_k$, $1\leq k\leq \ep$, we have that, for any $0\leq k\leq \io$, the set $B\cap F_{k}$, resp. $B\s F_k$,  is a basis of $M(F_{k})$, resp. $M/F_k$, as it is obtained by union of some of these former bases.

%
%\begin{color}{purple}
%***dessous : preuve de $\Int(\F)=\{a_1,\dots, a_\io\}$, peut etre suffisant pour ABG2***
%\end{color}

%\red{FAIRE LEMME pour ce "second..." ? car bien reutilise ensuite dans preuve}

Second, let us prove that $\Int(\F)=\{a_1,\dots, a_\io\}$. 
%where $a_k=\min(F_k\s F_{k-1})$, $1\leq k\leq \io$.
%and $\Ext(\F)=\{a'_1,\dots,a'_\ep\}$, where $a'_k=\min(F'_{k-1}\s F'_{k})$, $1\leq k\leq \ep$.
%\red{NB : serait utile aussi pour ABG2 !}

Let $b\in \Int(\F)$. By definition, $b\in B$ and $b=\min(C^*(B;b))$.
By assumption on the sequence $\S$, %$\S$, 
$b$ is an element of a minor $N$ of $M$ induced by this sequence $\S$: either $N=M_k$ for some $1\leq k\leq \io$ or $N=M'_k$ for some $1\leq k\leq \ep$.
In any case, $b$ is an element of the basis $B_N$ induced by $B$ in $N$: either $B_N=B_k$ if $N=M_k$, or $B_N=B'_k$ if $N=M'_k$.
Moreover, since $N$ is of type $M(G)/F$ and its basis $B_N$ of type $B\cap (G\s F)$, we have by Property \ref{pty:fund_graph_flat} that
$C^*_N(B_N;b)$ is obtained from $C^*_M(B;b)$ by removing elements not in the ground set of $N$.
%\red{ne devrait on pas montrer cette pty pour $M/A\s B$ ??? ou detailler ceci ? BOF...}
So $b=\min(C^*_N(B_N;b))$, so $b$ is internally active in $N$. By assumption on the  sequence $\S$ this implies that $b=a_k$ for some $1\leq k\leq \io$.
Hence $\Int(\F)\subseteq\{a_1,\dots, a_\io\}$.

Conversely, let $1\leq k\leq \io$.
By assumption on the sequence $\S$, we have $a_k=\min(C^*_{M_k}(B_k;a_k))$.
As above, by Property \ref{pty:fund_graph_flat}, we have $C^*_{M_k}(B_k;a_k)=C^*_M(B;a_k)\cap (F_k\s F_{k-1})$. 
Let $e=\min(C^*_M(B;a_k))$ and assume that $e<a_k$.
Since $\S$ is a filtration, by Definition \ref{def:general-filtration}, the sequence $a_j=\min(F_j\s F_{j-1})$ is increasing with $j$. Hence 
$a_k=\min(E\s F_{k-1})$. Hence $e\in F_{k-1}$.
On the other hand, by properties of matroid contraction, since $B\s F_{k-1}$ is a basis of $M/F_{k-1}$, we have $C^*_{M}(B;b)\cap F_{k-1}=\emptyset$, which is a contradiction with $e\in F_{k-1}$.
So we have $e=a_k$. So $a_k\in \Int(\F)$ and we have proved $\Int(\F)\supseteq\{a_1,\dots, a_\io\}$.
Finally, we have proved $\Int(\F)=\{a_1,\dots, a_\io\}$. 
%Notice that, by duality (i.e. applying this result to $M^*$), we have also $\Ext(\F)=\{a'_1,\dots,a'_\ep\}$.
%\purple{apparemment on ne sert du fait que les $a_k$ sont croissants que ci-desus. SI on s'en passe ca pourrait montrer la conjecture plus generale ci-dessous sans contraitne d'ordre}
%%\red{attention notation $\Int(\F)$ et $\Int_M(B)$. pas genant comme c'est dit dessous}
%
%\begin{color}{purple}
%***FIN de preuve de $\Int(\F)=\{a_1,\dots, a_\io\}$, peut etre suffisant pour ABG2***
%\end{color}

Third, let us prove that for every $k$, $0\leq k\leq \io$, we have
 $\Int_{M(F_k)}(B\cap F_k)=\{a_1,\dots, a_k\}$, resp. 
  $\Int_{M/F_{k}}(B\s F_{k})=\{a_{k+1},\dots, a_\io\}$.
  
  We obtain this result by directly applying the above result (that is $\Int_M(B)=\{a_1,\dots, a_\io\}$) in the minor $M(F_k)$, resp. $M/F_k$, of $M$.
Precisely, let $0\leq k\leq \io$. 
As noticed above,
the set $B\cap F_{k}$, resp. $B\s F_k$, is a basis of $M(F_{k})$, resp. $M/F_k$.
Obviously, by Definition \ref{def:general-filtration}, we have that 
$\emptyset=F'_\ep\subset \ldots\subset F'_0= F_c = F_0\subset \ldots \subset F_{k}$,
resp.  
$\emptyset=F_c\subset F_{k+1}\s F_k \ldots\subset F_\io\s F_k$,
is a filtration of $F_{k}$, resp. $E\s F_k$, and that it satisfies the properties given in the proposition statement (as the induced minors are minors also induced by $\S$, that is by $\emptyset=F'_\ep\subset \ldots\subset F'_0= F_c = F_0\subset \ldots \subset F_{\io}$).
The set of smallest elements of successive differences of sets in the sequence is 
$\{a_1,\dots, a_k\}$, resp. 
  $\{a_{k+1},\dots, a_\io\}$.
So we can apply the same reasoning as above to the minor $M(F_k)$, resp. $M/F_k$, of $M$, and we obtain the same result.

Fourth, let us prove that, for every $k$, $0\leq k\leq \io$, we have
 $\Ext_{M(F_k)}(B\cap F_k)=\Ext_M(B)$, resp. 
  $\Ext_{M/F_{k}}(B\s F_{k})=\emptyset$.
%In particular, $\Ext_M(B)=\{a'_1,\dots,a'_\ep\}$.
  
Applying the above result (that is $\Int_M(B)=\{a_1,\dots, a_\io\}$) in the dual $M^*$ of $M$, we directly have $\Ext_M(B)=\{a'_1,\dots,a'_\ep\}$.
Now, as above, let us apply this last result (that is $\Ext_M(B)=\{a'_1,\dots,a'_\ep\}$) in the minor  $M(F_k)$, to the filtration
$\emptyset=F'_\ep\subset \ldots\subset F'_0= F_c = F_0\subset \ldots \subset F_{k}$ of $F_k$.
We obtain  $\Ext_{M(F_k)}(B\cap F_k)=\{a'_1,\dots,a'_\ep\}=\Ext_M(B)$.
And let us apply the same result in the minor $M/F_k$
$\emptyset=F_c\subset F_{k+1}\s F_k \ldots\subset F_\io\s F_k$,
to the filtration of $E\s F_k$. We obtain  $\Ext_{M/F_{k}}(B\s F_{k})=\emptyset$.

Finally, we have proved that, 
for every $k$, $0\leq k\leq \io$, 
%\red{attention indices, prendre bien tous comme ici, verifier aileurs}
 and denoting $X=\{a_{k+1},\dots,a_\io\}$,
the following conditions are satisfied:
\vspace{-2mm}
\begin{enumerate}[(i)]
\partopsep=0mm \topsep=0mm \parsep=0mm \itemsep=0mm
\item $B\cap F_k$ is a basis of $M(F_k)$
\item  $\Int_{M(F_k)}(B\cap F_k)=\Int_M(B)\s X$,\par 
\item  $\Ext_{M(F_k)}(B\cap F_k)=\Ext_M(B)$,\par
\item  $\Int_{M/F_k}(B\s F_k)=X$,\par
\item  $\Ext_{M/F_k}(B\s F_k)=\emptyset$.
\end{enumerate}
\vspace{-2mm}
By uniqueness in Proposition \ref{prop:dec_base}, this implies $F_k=E\s \AA(X)=E\s \AA(\{a_{k+1},\dots,a_\io\})$, which matches Definition \ref{def:act-seq-dec-fund-graph} of the active filtration.
%\red{utiliser notation $\tilde F$ du coup ?}

At last, by duality, we also have,
for every $k$, $0\leq k\leq \ep$, denoting $X=\{a'_{k+1},\dots,a'_\ep\}$,
 that $F'_k=\AA(\{a_{k+1},\dots,a_\io\})$, which matches Definition \ref{def:act-seq-dec-fund-graph} of the active filtration (notice that, in particular, $F_0=F'_0$).
So finally the filtration $(F'_\ep, \ldots, F'_0, F_c , F_0, \ldots, F_\io)$ is the active filtration of $\F$.
\end{proof}

%\begin{color}{purple}
%CONJECTURE :
%on a prop precedente pour ordre quelconque des sous-ensembles, sous laforme
%B a activites i,j si et seulement si il existe un suite decomposante generale (sanscotnrainte que suite des \min est croissante) qui induit des bases uniactive internal/external en nombre i,j
%(ATTENTION : pour orientations prendre l'union de tous les cocirs ayant pour \min un element donne pet donner E, donc pas pareil!)
%\end{color}

\begin{observation}
\label{lem:dec-seq-bas-observation-suite}
Let us continue and refine Observation \ref{lem:dec-seq-bas-observation}.
Let $\emptyset= F'_\ep\subset...\subset F'_0=F_c=F_0\subset...\subset F_\io= E$ be the active filtration of the basis $B$ of $M$.
And let $F$ and $G$ be two subsets in this sequence such that $F\subseteq G$. Then, by Theorem \ref{thm:unique-dec-seq-bas},
the active filtration of the basis $B\cap G\s F$ of $M(G)/F$ is obtained from the subsequence with extremities $F$ and $G$ (i.e. $F\subset \dots \subset G$) of the active filtration of $B$ by  subtracting $F$ from each subset of the subsequence (with $F_c\s F$ as cyclic flat).
In particular, the subsequence ending with $F$ (i.e. $\emptyset\subset \dots \subset F$) yields the active filtration of $B\cap F$ in $M(F)$, and the subsequence beginning with $F$ (i.e. $F\subset \dots \subset E$) yields the active filtration of $B\s F$ in $M/F$ by subtracting $F$ from each subset.
\emevder{peut-etre mieux de faire un belle proposition ou un corollaire de cette observation ?}
\end{observation}

%\red{******************************* DESSOUS TEMP ************************************}

%\red{ajouter dessin de tableaux, voir expose tutte}

%\begin{figure}[htp]
%{\scalebox{1}
%{\includegraphics[width=6cm]{figures/K4exbasedecomp}}
%}
%\caption{\red{ac ompleter}}
%\label{fig:K4exbasedecomp}
%\end{figure}

%\def\rbullet{\red{\triangle}}
%\def\bbullet{\blue{\square}}
%\def\gbullet{\gray{\bullet}}

%\def\rbullet{\red{\medbullet}}
%\def\bbullet{\blue{\bigstar}}
%\def\gbullet{\gray{\times}}
%\def\gbullet{{\times}}

%
%\def\rbullet{\red{\mdblkdiamond}}
%\def\bbullet{\blue{\blacksquare}}
%\def\gbullet{\gray{\times}}

%\def\rbullet{\red{\medbullet}}
%\def\gbullet{{\times}}
%\def\bbullet{\blue{\blacksquare}}

%
%\def\rbullet{\red{\bullet}}
%\def\bbullet{\blue{\bullet}}
%\def\gbullet{\gray{\bullet}}

%\def\rbullet{\red{\bullet}}
\def\bbullet{\blue{\blacksquare}}

\def\rbullet{\tikz\draw[red,fill=red] (0,0) circle (.8ex);}
\def\gbullet{{\otimes}}

\def\fullsquare{\tikz\draw[blue,fill=blue, opacity=0.5] (0,0) rectangle (0.22,0.22);}
\def\littlefullsquare{\tikz\draw[blue,fill=blue, opacity=0.5] (0,0) rectangle (0.13,0.13);}

\def\bbullet{\blue{\fullsquare}}
\def\taille{1}

\def\bgbullet{\blue{\littlefullsquare}}

\def\rgbullet{\red{\otimes}}

%\red{changer symboles pour noir et blanc : gros ronds rouges, gros ronds vides bleus, et croix grises}

%\red{redonner memes exemples signes dans algo basori que dans section decomp bases}

%\begin{figure}[htp]
%\brown{Tableau of a basis $135$ with activities $(1,0)$}
%
%%\center{
%\renewcommand{\arraystretch}{\taille}
%%\small
%\begin{tabular}{|c|c|c|c|c|c|c|}
%\hline
% & \brown{$C^*_1$} & \tt{2} & \brown{$C^*_3$}&\tt{4}&\brown{$C^*_5$}&\tt{6} \\
%\hline
%\tt{1}& $\bbullet$ && &&&\\
%\brown{$C_2$}& $\bbullet$ &$\bbullet$ & $\bbullet$ &&&\\
%\tt{3}& && $\bbullet$ &&&\\
%\brown{$C_4$}& $\bbullet$ && &$\bbullet$ &$\bbullet$& \\
%\tt{5}&&& &&$\bbullet$ & \\
%\brown{$C_6$}& &&$\bbullet$&&$\bbullet$& $\bbullet$ \\
%\hline
%\end{tabular}
%%}
%%\end{figure}
%

%\newpage

%\red{CLARFIIER BLEU POUR IMPRESSION EN NOIR ET BLANC}

\begin{figure}[!h]

   \begin{minipage}[c]{.46\linewidth}
   \centering
   	\scalebox{1.3}{
      \includegraphics{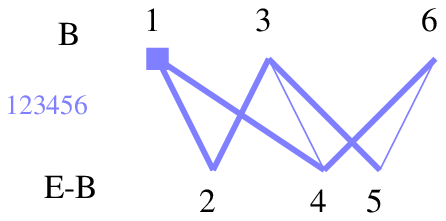}
      }
   \end{minipage}
   \hfill
      \begin{minipage}[c]{.46\linewidth}
      \centering
      \renewcommand{\arraystretch}{\taille}
	%\small
		\begin{tabular}{|c|c|c|c|c|c|c|}
		\hline
		 & \coltab{$C^*_1$} & \ptt{2} & \coltab{$C^*_3$}&\ptt{4}&\ptt{5}&\coltab{$C^*_6$} \\
		\hline
		\ptt{1}& $\bbullet$ && &&&\\
		\coltab{$C_2$}& $\bbullet$ &$\bbullet$ & $\bbullet$ &&&\\
		\ptt{3}& && $\bbullet$ &&&\\
		\coltab{$C_4$}& $\bbullet$ &&$\bgbullet$ &$\bbullet$ &&$\bbullet$ \\
		\coltab{$C_5$}& &&$\bbullet$&&$\bbullet$& $\bgbullet$ \\
		\ptt{6}&&& &&&$\bbullet$  \\
		\hline
		\end{tabular}												
   \end{minipage}
\caption{Fundamental graph/tableau of basis $136$ with activities $(1,0)$ and active partition $E=123456$.}
\label{fig:exbasedecomp-136}
\end{figure}

\begin{figure}[!h]
   \begin{minipage}[c]{.46\linewidth}
   \centering
   	\scalebox{1.3}{
      \includegraphics{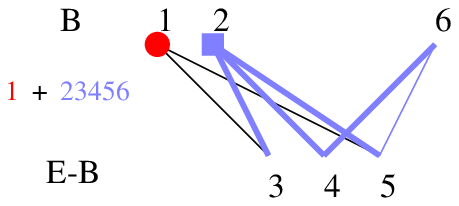}
      }
   \end{minipage}
   \hfill
      \begin{minipage}[c]{.46\linewidth}
      \centering
      \renewcommand{\arraystretch}{\taille}
	%\small
		\begin{tabular}{|c|c|c|c|c|c|c|}
		\hline
		 & \coltab{$C^*_1$} & \coltab{$C^*_2$} & \ptt{3}&\ptt{4}&\ptt{5}&\coltab{$C^*_6$} \\
		\hline
		\ptt{1}& $\rbullet$ && &&&\\
		\ptt{2}& & $\bbullet$& &&&\\
		\coltab{$C_3$}& $\gbullet$ &$\bbullet$ & $\bbullet$ &&&\\
		\coltab{$C_4$}&  &$\bbullet$&&$\bbullet$&&$\bbullet$ \\
		\coltab{$C_5$}&$\gbullet$&$\bbullet$& &&$\bbullet$ &$\bgbullet$ \\
		\ptt{6}& &&&&& $\bbullet$ \\
		\hline
		\end{tabular}												
   \end{minipage}
\caption{Fundamental graph/tableau of basis $126$ with activities $(2,0)$ and active partition $E=1+23456$.}
\label{fig:exbasedecomp-126}
\end{figure}

%\coltab{Tableau of a basis $136$ with activities $(1,0)$}
%
%%\center{
%\renewcommand{\arraystretch}{\taille}
%%\small
%\begin{tabular}{|c|c|c|c|c|c|c|}
%\hline
% & \coltab{$C^*_1$} & \ptt{2} & \coltab{$C^*_3$}&\ptt{4}&\ptt{5}&\coltab{$C^*_6$} \\
%\hline
%\ptt{1}& $\bbullet$ && &&&\\
%\coltab{$C_2$}& $\bbullet$ &$\bbullet$ & $\bbullet$ &&&\\
%\ptt{3}& && $\bbullet$ &&&\\
%\coltab{$C_4$}& $\bbullet$ &&$\gbullet$ &$\bbullet$ &&$\bbullet$ \\
%\coltab{$C_5$}& &&$\bbullet$&&$\bbullet$& $\gbullet$ \\
%\ptt{6}&&& &&&$\bbullet$  \\
%\hline
%\end{tabular}
%%}
%%\end{figure}

%{
%\coltab{Tableau of a basis $126$ with activities $(2,0)$:}
%
%%\center{
%\renewcommand{\arraystretch}{\taille}
%%\small
%\begin{tabular}{|c|c|c|c|c|c|c|}
%\hline
% & \coltab{$C^*_1$} & \coltab{$C^*_2$} & \ptt{3}&\ptt{4}&\ptt{5}&\coltab{$C^*_6$} \\
%\hline
%\ptt{1}& $\rbullet$ && &&&\\
%\ptt{2}& & $\bbullet$& &&&\\
%\coltab{$C_3$}& $\gbullet$ &$\bbullet$ & $\bbullet$ &&&\\
%\coltab{$C_4$}&  &$\bbullet$&&$\bbullet$&&$\bbullet$ \\
%\coltab{$C_5$}&$\gbullet$&$\bbullet$& &&$\bbullet$ &$\bbullet$ \\
%\ptt{6}& &&&&& $\bbullet$ \\
%\hline
%\end{tabular}
%%}
%}

\begin{figure}[!h]
   \begin{minipage}[c]{.46\linewidth}
   \centering
   	\scalebox{1.3}{
      \includegraphics{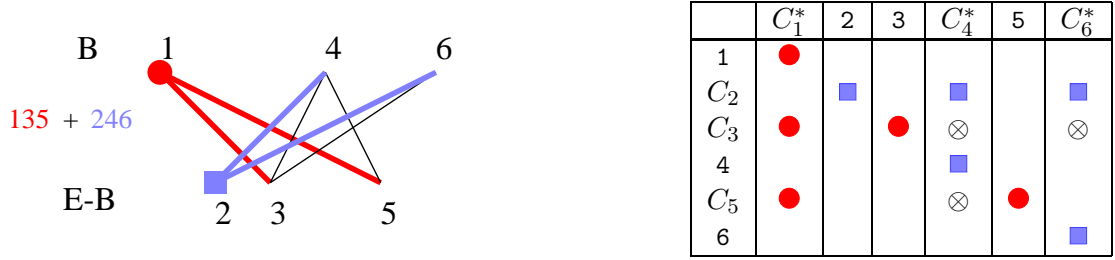}
      }
   \end{minipage}
   \hfill
      \begin{minipage}[c]{.46\linewidth}
      \centering
      \renewcommand{\arraystretch}{\taille}
	%\small
		\begin{tabular}{|c|c|c|c|c|c|c|}
		\hline
		 & \coltab{$C^*_1$} & \ptt{2} &\ptt{3} &\coltab{$C^*_4$}&\ptt{5}&\coltab{$C^*_6$} \\
		\hline
		\ptt{1}& $\rbullet$ && &&&\\
		\coltab{$C_2$}& &$\bbullet$ & &$\bbullet$&&$\bbullet$\\
		\coltab{$C_3$}&$\rbullet$ && $\rbullet$  &$\gbullet$ &&$\gbullet$ \\
		\ptt{4}&  && &$\bbullet$ && \\
		\coltab{$C_5$}&$\rbullet$ &&  &$\gbullet$&$\rbullet$ & \\
		\ptt{6}& &&&&& $\bbullet$ \\
		\hline
		\end{tabular}											
   \end{minipage}
\caption{Fundamental graph/tableau of basis $146$ with activities $(1,1)$ and active partition $E=135+246$.}
\label{fig:exbasedecomp-146}
\end{figure}

%{
%\coltab{Tableau of a basis $146$ with activities $(1,1)$:}
%
%%\center{
%\renewcommand{\arraystretch}{\taille}
%%\small
%\begin{tabular}{|c|c|c|c|c|c|c|}
%\hline
% & \coltab{$C^*_1$} & \ptt{2} &\ptt{3} &\coltab{$C^*_4$}&\ptt{5}&\coltab{$C^*_6$} \\
%\hline
%\ptt{1}& $\rbullet$ && &&&\\
%\coltab{$C_2$}& &$\bbullet$ & &$\bbullet$&&$\bbullet$\\
%\coltab{$C_3$}&$\rbullet$ && $\rbullet$  &$\gbullet$ &&$\gbullet$ \\
%\ptt{4}&  && &$\bbullet$ && \\
%\coltab{$C_5$}&$\rbullet$ &&  &$\gbullet$&$\rbullet$ & \\
%\ptt{6}& &&&&& $\bbullet$ \\
%\hline
%\end{tabular}
%%}
%}

\begin{figure}[!h]
   \begin{minipage}[c]{.46\linewidth}
   \centering
   	\scalebox{1.3}{
      \includegraphics{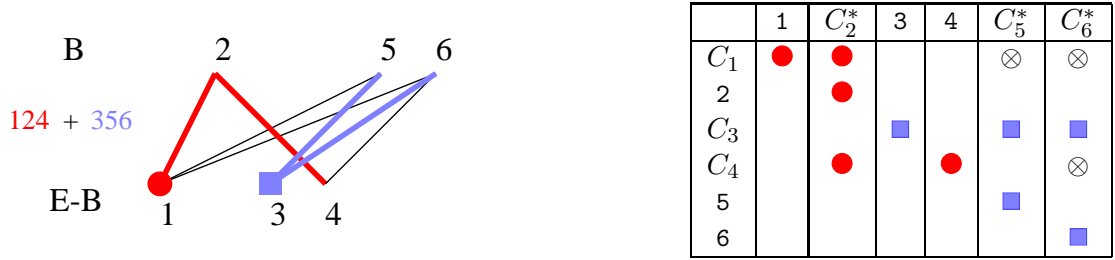}
      }
   \end{minipage}
   \hfill
      \begin{minipage}[c]{.46\linewidth}
      \centering
      \renewcommand{\arraystretch}{\taille}
	%\small
		\begin{tabular}{|c|c|c|c|c|c|c|}
		\hline
		 & \ptt{1} & \coltab{$C^*_2$} &\ptt{3} &\ptt{4}&\coltab{$C^*_5$}&\coltab{$C^*_6$} \\
		\hline
		\coltab{$C_1$}& $\rbullet$ &$\rbullet$& &&$\gbullet$&$\gbullet$\\
		\ptt{2}& &$\rbullet$ & &&&\\
		\coltab{$C_3$}& && $\bbullet$  &&$\bbullet$ &$\bbullet$ \\
		\coltab{$C_4$}&  &$\rbullet$ & &$\rbullet$ &&$\gbullet$  \\
		\ptt{5}& &&  &&$\bbullet$ & \\
		\ptt{6}& &&&&& $\bbullet$ \\
		\hline
		\end{tabular}									
   \end{minipage}
\caption{Fundamental graph/tableau of basis $256$ with activities $(0,2)$ and active partition $E=124+356$.}
\label{fig:exbasedecomp-256}
\end{figure}

%\red{expliquer figures}
%
%Let us also observe that, in Theorem \ref{th:dec_base}, for $1\leq k\leq \io$, if $M(F_k)/F_{k-1}$ is an isthmus, then $B_k$ equals this isthmus, and that, for $1\leq k\leq \ep$, if $M(F'_{k-1})/F'_{k}$ is a loop, then $B'_k=\emptyset$.

\begin{example}
\label{ex2}
{\rm
Figures \ref{fig:exbasedecomp-136}, \ref{fig:exbasedecomp-126}, \ref{fig:exbasedecomp-146}, \ref{fig:exbasedecomp-256} show active decompositions/partitions of some fundamental graphs/tableaux. They illustrate also bases of $K_4$ from Example \ref{ex1} and Figure \ref{fig:K4exbase256}.
In the graphs: the full circles and full squares show the internal/external active elements, and the bold paths of edges connected to these elements show the active partition (restricting the fundamental graph to the subsets of edges forming these parts yield uniactive fundamental graphs); and the light edges are not involved in the construction.
In the tableaux: the full circles and full squares show the active partition (restricting the fundamental tableau to the subsets of entries forming these parts yield uniactive fundamental tableaux);
and the circled crosses and the little squares are not involved in the construction (circled crosses disappear when restricting to tableaux induced by the active partition).
The fundamental circuits and cocircuits are also indicated at the beginning of concerned rows and columns of the tableaux.
}
\end{example}

\begin{thm} %[\cite{Gi02, AB2}]  %[Main Theorem: Active Decomposition of Bases] %{50}
\label{th:dec_base}
%\red{ATTENTION disjoint union mais parties peuvent etre vides, ca ne change rien, a remarquer quand meme ?}
%Let $E$ be a linearly ordered set.
Let $M$ be a matroid on a linearly ordered set $E$.
%
%nd ${\cal B}$ the set of bases of $M$.\par
$$
%{\cal B}
\Bigl\{\ \text{bases of }M\ \Bigr\}\ 
=\biguplus_
{\substack{
\emptyset=F'_\ep\subset...\subset F'_0=F_c\\
F_c=F_0\subset...\subset F_\iota=E\\
\hbox{\small connected filtration of $M$}
}}
\Bigl\{B'_1\plus...\plus B'_\ep\plus B_1\plus...\plus B_\iota
\mid$$
%$$\forall 1\leq k\leq \epsilon, \
%B'_k \hbox{ base with activities }(0,1)\hbox{ of } M(F'_{k-1})/F'_{k},$$
$$\text{for all }1\leq k\leq \ep, \
B'_k \hbox{ base of }M(F'_{k-1})/F'_{k}\text{ with $\io(B'_k)=0$ and $\ep(B'_k)=1$,}$$
%$$\forall 1\leq k\leq \iota, \
%B_k \hbox{ base with activities }(1,0)\hbox{ of } M(F_k)/F_{k-1} \}
%$$
$$\text{for all }1\leq k\leq \io, \
B_k \hbox{ base of }M(F_k)/F_{k-1}\text{ with $\io(B_k)=1$ and $\ep(B_k)=0$}\Bigr\}$$
%\bigskip
%\red{with the above notations...}
With the above notations and $B=B'_1\plus...\plus B'_\ep\plus B_1\plus...\plus B_\iota$,
we then have:
$$\Int(B)=\cup_{1\leq k\leq \iota} \min(F_k\s F_{k-1})=\cup_{1\leq k\leq \iota} \Int(B_k),$$
$$\Ext(B)=\cup_{1\leq k\leq \ep} \min(F'_{k-1}\s F'_{k})=\cup_{1\leq k\leq \ep} \Ext(B'_k).$$
Moreover, the connected filtration associated to the basis $B$ in the right-hand side of the equality  is the active filtration of (the fundamental graph of) $B$.
\end{thm}

\begin{proof}
This theorem simply consists in applying Theorem \ref{thm:unique-dec-seq-bas} to all bases at the same time.
Let $B$ be a basis of $M$ with fundamental graph $\F$.
By Theorem \ref{thm:unique-dec-seq-bas}, the active filtration of $\F$ induces exactly the partition $B=B'_1\plus...\plus B'_\ep\plus B_1\plus...\plus B_\iota$ as stated in the present theorem. By properties of these bases, we have 
$\cup_{1\leq k\leq \iota} \min(F_k\s F_{k-1})=\cup_{1\leq k\leq \iota} \Int(B_k)$
and $\cup_{1\leq k\leq \ep} \min(F'_{k-1}\s F'_{k})=\cup_{1\leq k\leq \ep} \Ext(B'_k)$.
And, by definition of the active filtration and Lemma \ref{lem:pty-act-dec-seq-bas}, we have $\Int(B)=\cup_{1\leq k\leq \iota} \min(F_k\s F_{k-1})$ and $\Ext(B)=\cup_{1\leq k\leq \ep} \min(F'_{k-1}\s F'_{k})$.

Conversely, let $B'_1,..., B'_\ep, B_1,..., B_\iota$ as stated in the present theorem for a given connected filtration of $M$. 
Obviously, and as shown in the proof of Theorem \ref{thm:unique-dec-seq-bas},
we have that $B=B'_1\plus...\plus B'_\ep\plus B_1\plus...\plus B_\iota$ 
is a basis of $M$.
Furthermore, by uniqueness property in Theorem \ref{thm:unique-dec-seq-bas}, we have that the filtration is the active filtration of the fundamental graph of $B$, which implies as above that $\Int(B)=\cup_{1\leq k\leq \iota} \min(F_k\s F_{k-1})=\cup_{1\leq k\leq \iota} \Int(B_k),$
and $\Ext(B)=\cup_{1\leq k\leq \ep} \min(F'_{k-1}\s F'_{k})=\cup_{1\leq k\leq \ep} \Ext(B'_k).$
\end{proof}

\begin{remark}
\label{rk:important-uniqueness}
\rm
\emevder{environnement remark pour cette rk ?}%
%Let us also anticipate and make the following remark.
One sees how the uniqueness result in Theorem \ref{thm:unique-dec-seq-bas} is important.
The easier result, without uniqueness, contained in this theorem just states that the bases induced in the active minors induced by the active filtration are uniactive internal/external. 
From this weaker result, one could derive
a weaker version of  Theorem \ref{th:dec_base} above with a union instead of a disjoint union, and then a weaker version of the Tutte polynomial formula in Theorem \ref{th:tutte} with an inequality  instead of an equality.
It is the uniqueness that allows to state Theorems \ref{th:dec_base} and \ref{th:tutte} as they are.
\end{remark}

\begin{cor}[\cite{EtLV98}] %{80}
\label{th:EtLV98}
Let $M$ be a matroid on a linearly ordered set $E$.
\begin{eqnarray*}
\{\text{bases of $M$}\}=\biguplus_{\substack{F_c \hbox{ cyclic } \\ \hbox{ flat of } M}}
\{B' \plus B \mid & B' \hbox{ base of } M(F_c)\hbox{ with internal activity $0$, }\\
& B \hbox{ base of } M/F_c\text{ with external activity $0$}\}
\end{eqnarray*}
\end{cor}

\begin{proof}
Direct by Observation \ref{lem:dec-seq-bas-observation} and Theorem \ref{th:dec_base} applied to decompose the set of bases of $M$, the set of external bases of $M(F_c)$, and the set of internal bases of $M/F_c$, for all cyclic flats $F_c$ of $M$.
\end{proof}

\begin{example}
{\rm
Figure \ref{fig:tabK4} shows the decomposition of bases of $K_4$, provided by Theorem \ref{th:dec_base}, completing Example \ref{ex1}, Example \ref{ex2}, and Figures \ref{fig:K4exbase256}, \ref{fig:exbasedecomp-136}, \ref{fig:exbasedecomp-126}, \ref{fig:exbasedecomp-146}, \ref{fig:exbasedecomp-256}.
}
\end{example}

\def\fcyc #1{\fbox{\hbox{#1}}}
\begin{figure}[H]
\begin{center}
\begin{tabular}{|c|c|c|c|c|c|}
\hline
Active filtrations & Active partitions &  {\small Uniactive bases of minors} & $\Ext$ & $\Int$ & Bases  \\
\hline
$\fcyc{\O}\subset 1\subset 123\subset E$ & $1+23+456$ & 1+2+4 & $\emptyset$& 124  & 124 \\
$\fcyc{\O}\subset 1\subset E$ & $1+23456$ & $1+26$& $\emptyset$& 12&  126 \\
$\fcyc{\O}\subset 145\subset E$ & $145+236$  & $15+2$ & $\emptyset$ & 12 & 125 \\
$\fcyc{\O}\subset 123\subset E$ & $123+456$& $13+4$& $\emptyset$ & 14  & 134 \\
$\fcyc{\O}\subset E$ & $123456$  & $135$&$\emptyset$ & 1  & 135 \\
$\fcyc{\O}\subset E$ & $123456$ & $136$&$\emptyset$ & 1&136 \\
$\emptyset\subset \fcyc{246}\subset E$ & $246+135$ & $46+1$& 2 &1 & 146 \\
$\emptyset\subset \fcyc{356}\subset E$ & $356+124$ & $56+1$&3 & 1& 156 \\
$\emptyset\subset \fcyc{123}\subset E$ & $123+456$ & $23+4$ &1&4& 234 \\
$\emptyset\subset \fcyc{145}\subset E$ & $145+236$ & $45+2$&1 &2& 245 \\
$\emptyset\subset \fcyc{E}$ & $123456$  & $235$& 1& $\emptyset$& 235 \\
$\emptyset\subset \fcyc{E}$ & $123456$  & $236$ & 1& $\emptyset$& 236 \\
$\emptyset\subset 356 \subset \fcyc{E}$ & $356+124$  & $56+2$ &13& $\emptyset$& 256 \\
$\emptyset\subset 246 \subset \fcyc{E}$ & $246+135$ & $46+3$ & 12& $\emptyset$&346 \\
$\emptyset\subset 23456\subset\fcyc{E}$ & $23456+1$ & $345+\emptyset$ &12& $\emptyset$& 345 \\
$\emptyset\subset 356\subset 23456\subset \fcyc{E}$ & $356+24+1$ & $56+4+\emptyset$ &123& $\emptyset$&456 \\
\hline
\end{tabular}
\caption{Table of all connected filtrations, bases, and related information, for the matroid $K_4$ from Figure \ref{fig:K4exbase256}, illustrating Theorem \ref{th:dec_base}. The cyclic flat of each connected filtration is boxed. Beware that, in this example, only the two trivial filtrations serve for more than one base, whereas, in general, a same filtration can obviously serve for numerous bases.}
\label{fig:tabK4}
\end{center}
\vspace{-0.5cm}
\end{figure}

\begin{prop}[Single-pass computation of the active partition of a matroid basis or a fundamental graph/tableau]
\label{prop:basori-partact}
%\red{ou proposition}
Let $M$ be a matroid on a linearly ordered set of elements $E=e_1<\ldots<e_n$.
Let $B$ be a base of $M$.
The algorithm below computes the active partition of $B$ as a mapping, denoted $\ass$, from $E$ to $\Int(B)\cup \Ext(B)$, that maps an element onto the smallest element of its part in the active partition of $B$. An element is called internal, resp. external, if its image is in $\Int(B)$, resp. $\Ext(B)$.
Hence the active partition of $B$ is $$\biguplus_{e\ \in\ \Int(B)\ \cup\ \Ext(B)}\ass^{-1}(e),$$ with external part given by $\ass^{-1}(\Ext(B))$.
%\red{definir application part dans corp du texte de section NON}
The algorithm consists in a single pass over $E$.
It only relies upon the fundamental graph/tableau (and can be equally applied to decompose a fundamental graph/tableau).
Note that the rules when $e_k\in B$ are dual to the rules when $e_k\not\in B$, and that the rules when $e_k$ is internal are dual to the rules when $e_k$ is external.

\begin{algorithme} \par
%\underbar{Direct computation of $Basori_M(B)$}

For $k$ from $1$ to $n$ do\par

\hskip 5mmif $e_k\in B$ then \par
\hskip 5mm\hskip 5mm if $e_k$ is internally active  w.r.t. $B$ then\par
\hskip 5mm\hskip 10mm $e_k$ is internal\par 
\hskip 5mm\hskip 10mm let $\ass(e_k):=e_k$\par 
\hskip 5mm\hskip 5mm otherwise\par

\hskip 5mm\hskip 10mm it there exists $c<e_k$ external in $C^*(B;e_k)$ then\par

\hskip 5mm\hskip 15mm $e_k$ is external\par
\hskip 5mm\hskip 15mm let $c$ $\in$ $C^*(B;e_k)$ with $c<e_k$, $c$ external and $\ass(c)$ the greatest possible \par
\hskip 5mm\hskip 15mm let $\ass(e_k):=\ass(c)$ \par

\hskip 5mm\hskip 10mm otherwise\par
\hskip 5mm\hskip 15mm $e_k$ is internal\par
\hskip 5mm\hskip 15mm let $c$ $\in$ $C^*(B;e_k)$ with $c<e_k$ and $\ass(c)$ the smallest possible\par
\hskip 5mm\hskip 15mm let $\ass(e_k):=\ass(c)$ \finsi\par

\hskip 5mmif $e_k\not\in B$ then \par
\hskip 5mm\hskip 5mm if $e_k$ is externally active w.r.t. $B$ then \par
\hskip 5mm\hskip 10mm $e_k$ is external\par
\hskip 5mm\hskip 10mm let $\ass(e_k):=e_k$\par
%and arbitrary choice $e_k\in A$ or $e_k\not\in A$
%\par
\hskip 5mm\hskip 5mm otherwise\par

\hskip 5mm\hskip 10mm if there exists $c<e_k$ internal in $C(B;e_k)$ then\par

\hskip 5mm\hskip 15mm $e_k$ is internal\par
\hskip 5mm\hskip 15mm let $c$ $\in$ $C(B;e_k)$ with $c<e_k$, $c$ internal and $\ass(c)$ the greatest possible\par
\hskip 5mm\hskip 15mm let $\ass(e_k):=\ass(c)$\par
\hskip 5mm\hskip 10mm otherwise\par

\hskip 5mm\hskip 15mm $e_k$ is external\par
\hskip 5mm\hskip 15mm let $c$ $\in$ $C(B;e_k)$ with $c<e_k$ and $\ass(c)$ the smallest possible\par
\hskip 5mm\hskip 15mm let $\ass(e_k):=\ass(c)$ \finsi\par

\end{algorithme}%
\end{prop}

\begin{proof}
Let us denote $a_1,\dots,a_\io$, resp. $a'_1,\dots,a'_\ep$, the set of internally, resp. externally, active elements of $B$, and  $(F'_\ep, \ldots, F'_0, F_c , F_0, \ldots, F_\io)$ the active filtration of $\F_M(B)$.
Before giving a formal proof, let us mention that this algorithm simply consists in a direct combination of the following algorithms,
each consisting in a single pass over $E$.
The second and third algorithms do not interfere in each other, since they consist in refinements of the two separate outputs of the first algorithm.
\begin{itemize}
\partopsep=0mm \topsep=0mm \parsep=0mm \itemsep=0mm
\item The algorithm of Corollary \ref{cor:dec-cyclic-flat} that computes the external/internal partition.
\item The algorithm of Lemma \ref{lem:act-closure-algo}, applied in priority to $X=\{a_\io\}$, then to $X=\{a_{\io-1},a_\io\}$, etc., then to $X=\{a_1,\dots, a_\io\}$. By this manner,  an element $e$ belonging to the internal part is mapped onto $a_i$, where $i$ is the greatest possible 
 such that $e\in \AA(a_{i},\dots,a_\io\}=E\s F_{i-1}$,
in order to have $e\in F_i\setminus F_{i-1}$, consistently with the definition of the active partition.
\item The algorithm of Lemma \ref{lem:act-closure-algo} applied in the dual, and in priority to $X=\{a'_\ep\}$, then to $X=\{a'_{\ep-1},a'_\ep\}$, etc., then to $X=\{a'_1,\dots, a'_\ep\}$. By this manner,  an element $e$ belonging to the external part is mapped onto $a'_i$, where $i$ is the greatest possible such that $e\in \AA(a'_{i+1},\dots,a'_\ep\}=F'_i$, in order to have
$e\in F'_{i-1}\setminus F'_{i}$, consistently with the definition of the active partition.
\end{itemize}
%The second and third algorithms do not interfere in each other, since they consist in refinements of the two separate outputs of the first algorithm.
%We let the reader check that,  in every case, the assignment given in the statement is correct.
%\red{non ! detailler, attention au "smallest"}

Now, let us verify precisely the assignments given in the algorithm statement.
Let $1\leq k\leq n$.
First, assume that $e_k\in B$ and $e_k$ is internally active, then, obviously, $\ass(e_k)=e_k$.

Second, assume that $e_k\in B$, $e_k$ is not internally active, and
every $c$ in $C^*(B;e_k)^<$ is internal.
Then, by Corollary \ref{cor:dec-cyclic-flat} , $e_k$ is internal.
Then, by Lemma \ref{lem:act-closure-algo}, we have $e_k\in\AA(a_{i},\dots,a_\io\}$ for all $i$ such that $C^*(B;e_k)^<\subseteq \AA(a_{i},\dots,a_\io\}$. Let $i$ be the greatest possible with this property. By definition of the active partition, we have $\ass(e_k)=a_i$, as we have $e_k\in \AA(a_{i},\dots,a_\io\}\s \AA(a_{i+1},\dots,a_\io\}= F_i\setminus F_{i-1}$.
Let $c\in C^*(B;e_k)^<$. We also have by definition of the active partition that $\ass(c)=a_j$ where $j$ is the greatest possible such that 
%$c\in \AA(a_{j},\dots,a_\io\}\s \AA(a_{j+1},\dots,a_\io\}= F_j\setminus F_{j-1}$.
$c\in \AA(a_{j},\dots,a_\io\}$.
Since  $c\in C^*(B;e_k)^<\subseteq  \AA(a_{i},\dots,a_\io\}$ by definition of $i$, we have $i\leq j$.
Assume now that  $c\not\in \AA(a_{i+1},\dots,a_\io\}$ (such a $c$ exists by definition of $i$). In this case, we have $i=j$, by definition of $j$.
We have proved that $a_i=\ass(e_k)$ is the smallest possible $a_j=\ass(c)$ over all $c\in C^*(B;e_k)^<$, which is exactly the assignment given in the algorithm.

Third, let us assume that
 assume that $e_k\not\in B$, $e_k$ is not externally active, and there exists $c\in C(B;e_k)^<$ which is internal.
 Then, by Corollary \ref{cor:dec-cyclic-flat} , $e_k$ is internal.
Then, by Lemma \ref{lem:act-closure-algo}, we have
 $e_k\in\AA(a_{i},\dots,a_\io\}$ for all $i$ such that there exists
 $c\in C(B;e_k)^< \cap \AA(a_{i},\dots,a_\io\}$. Let $i$ be the greatest possible with this property. By definition of the active partition, we have $\ass(e_k)=a_i$(as above). By definition of $c$, we have also $c\in \AA(a_{i},\dots,a_\io\}\s \AA(a_{i+1},\dots,a_\io\}= F_i\setminus F_{i-1}$, that is $\ass(c)=a_i=\ass(a_k)$.
 We have proved that $a_i=\ass(e_k)$ is the greatest possible $a_j=\ass(c)$ over all $c\in C(B;e_k)^<  \cap \AA(a_{1},\dots,a_\io\}$, which is exactly the assignment given in the algorithm.

The three other cases (where $e_k$ is external) are dual to the three above cases,
which completes the proof.
%Now, assume $e_k$ is not internally active and
%there exists $c<e_k$ external in $C^*(B;e_k)$.
%Then, by Corollary \ref{cor:dec-cyclic-flat} , $e_k$ is external.
%Then the greatest possible $i$  such that $e\in \AA(a'_{i+1},\dots,a'_\ep\}$
\end{proof}

\begin{proof}[Proof of Theorem \ref{th:tutte}]
%
%\red{ici o.m. a preprednre}
First, the fact that the two sets of sequences can be equally used directly  comes from Lemma \ref{lem:dec-seq-beta}.
Now, let us focus on the sum over connected filtrations of $M$. Recall that:

\vspace{-3mm}
\begin{itemize}[-]
\partopsep=-1mm \topsep=-1mm \parsep=-1mm \itemsep=-1mm
\item  for a matroid $M$ with at least two elements, there exists a uniactive internal basis, and equivalently a uniactive external basis, of $M$ if and only if $\beta(M)\not=0$, and equivalently $\beta^*(M)\not=0$;
\item
for a matroid $M$ with one element, $\beta(M)\not=0$ if and only if $M$ is an isthmus (which is an internal basis);
\item for a matroid $M$ with one element, $\beta^*(M)\not=0$ if and only if $M$ is a loop (which is an external basis).
\end{itemize}
\vspace{-2mm}
%\red{mettre en observation dans section tutte ce itemize ? ou dans preliminaires ?}
So we have that $\beta(M)\not=0$, resp. $\beta^*(M)\not=0$, if and only if $M$ has a uniactive internal, resp. external, basis.
Then  the formula given in the theorem is 
exactly the enumerative translation of Theorem \ref{th:dec_base}.

%directly given by Theorem \ref{th:dec-ori}, which constitutes a bijective proof of the result.
More precisely, consider the set of bases of $M$ with internal activity $\io$ and external activity $\ep$, whose cardinality is $b_{\io,\ep}$.
By Theorem \ref{th:dec_base}, using the same notations, this set bijectively corresponds to the set 
$\biguplus\ \bigl\{ B  \mid    
 \text{ for }  1\leq k\leq\io,\ B_k \hbox{ uniactive internal in } M_k,
\hbox{ and for }
  1\leq k\leq~\ep,$ $B'_k \hbox{ uniactive external in } M'_k  \bigr\}$
where the union is over all connected filtrations of $M$ with fixed $\io$ and $\ep$.
The cardinality of each part of this set is obviously 
%$\Bigl(\prod_{1\leq k\leq \io}
%\beta \bigl( M(F_k)/F_{k-1}\bigr)\Bigr)
% \ \Bigl(\prod_{1\leq k\leq \ep}\beta^* \bigl( M(F'_{k-1})/F'_{k}\bigr)\Bigr)$
$\Bigl(\prod_{1\leq k\leq \io}
\beta \bigl( M_k\bigr)\Bigr)
 \ \Bigl(\prod_{1\leq k\leq \ep}\beta^* \bigl( M'_k\bigr)\Bigr)$
 since $\beta$, resp. $\beta^*$, counts the number of uniactive internal, resp.  external, bases.
 By construction, the sum is the number of bases with internal activity $\io$ and external activity $\ep$, that is
 the coefficient $b_{\io,\ep}$ of $x^\io y^\ep$ % $t_{\io,\ep}$ 
in the Tutte polynomial, hence the result.
% 
%The second result, where the sum is over all abstract filtrations  of $E$, is equivalent to the first by Lemma \ref{lem:dec-seq-equiv}.
\end{proof}

%\red{dire dans cetion tute que prouve a t the very end of section...}

%
%%\begin{lemma}
%\begin{observation}
%\label{lem:dec-seq-bases-observation}
%\red{*** A ADAPTER A BASES, la copie de section tutte}
%Let $\emptyset= F'_\ep\subset...\subset F'_0=F_c=F_0\subset...\subset F_\io= E$ be a filtration of an ordered matroid $M$. We have:
%\begin{itemize}
%\item $\emptyset= E\s F_\io\subset...\subset E\s F_0=E\s F_c=E\s F'_0\subset...\subset E\s F'_\ep= E$ is a filtration of $M^*$, for the cyclic-flat $E\s F_c$ of $M^*$;
%\item $\emptyset= F'_\ep\subset...\subset F'_0=F_c=F_c$ is a filtration of $M(F_c)$, for the cyclic-flat $F_c$ of $M(F_c)$;
%\item $\emptyset=\emptyset= F_0\s F_c\subset...\subset F_\io\s F_c=E\s F_c$ is a filtration of $M/F_c$, for the cyclic-flat $\emptyset$ of $M/F_c$;
%\end{itemize}
%\end{observation}
%%\end{lemma}
%
%\begin{proof}
%Obvious by the definition.
%\end{proof}
%%%%%%%%%%%%%%%%%%%%%%%%%%%%%%%%%%%%%%%%%%%%%%%%%%%%

\emevder{VERIF REFS}%

%%%%%%%%%%%%%%%%%%%%%%%%%%%%%%%%%%%%%%%%%%%%%%%%%

%\red{verfiier si refs utilisees}

%%%%%%%%%%%%%%%%%%%%%%%%%%%%%%%%%%%%%%%%%%%%%%%%%%%%%%%%%%%%

%\bigskip
%\newpage
%%%%%%%%%%%%%%%%%%%%%%%%%%%%%%%%%%%%%%%%%%%%%%%%%%%%%%%%%%%%%

%\begin{bibliomatroidy}
%\noindent{\bf References.}

%\vspace{-0.2cm}

%\bibliomatroidystyle{plain}
\bibliographystyle{amsplain}

%\nocite{*}   %%% pour faire apparaitre toutes refs meme non citees

%\bibliomatroidy{AB-biblio}    %%% pour consutrire biblio a partir de base de donnees (attention : ordre non modifiable... donc j'ai modifie ensuite le .bbl)

%\red{mettre refs de active bijection ??? miennes kekzunes, autres ??}

\providecommand{\bysame}{\leavevmode\hbox to3em{\hrulefill}\thinspace}
\providecommand{\MR}{\relax\ifhmode\unskip\space\fi MR }
%\providecommand{\MRhref}[2]{%
%  \href{http://www.ams.org/mathscinet-getitem?mr=#1}{#2}
%}
\providecommand{\href}[2]{#2}

%\end{bibliomatroidy}
%%%%%%%%%%%%%%%%%%%%%%%%%%%%%%שש

\end{document}